\documentclass[12pt]{amsart}
\usepackage[top=1in, bottom=1in, left=1in, right=1in]{geometry}
\usepackage{amsfonts}
\usepackage{amsmath}
\usepackage{comment}
\usepackage{amssymb}
\usepackage{bbm}
\usepackage{comment}
\usepackage{mathrsfs}
\numberwithin{equation}{section}
\usepackage{times}

\usepackage{siunitx}
\usepackage[usenames,dvipsnames]{color}
\usepackage{comment}
\usepackage{xcolor}
\usepackage{mathtools}
\usepackage{bm}
\usepackage{esvect}
\usepackage{hyperref}
\hypersetup{
    colorlinks = true,
linkcolor={black},
urlcolor={blue},
citecolor={blue},    
urlcolor = {blue},
citebordercolor = {0.33 .58 0.33},
 linkbordercolor = {0.99 .28 0.23},
 breaklinks=true}

\newcommand{\kommentar}[1]{}

\newcommand{\R}{\mathbb{R}}

\newcommand{\N}{\mathbb{N}}
\newcommand{\Z}{\mathbb{Z}}

\newtheorem{thm}{Theorem}[section]

\newtheorem{coro}[thm]{Corollary}
\newtheorem{lem}[thm]{Lemma}
\newtheorem{rem}[thm]{Remark}

\newcommand{\dd}{\;\mathrm{d}}

\newcommand{\balp}{\bm{\alpha}}
\newcommand{\cE}{\mathcal{E}}
\newcommand{\cM}{\mathcal{M}}

\newcommand{\bh}{\mathbf{h}}

\newcommand{\fM}{\mathfrak{M}}
\newcommand{\fS}{\mathfrak{S}_{f,s}}
\newcommand{\fm}{\mathfrak{m}}

\newcommand{\cA}{\mathcal{A}}
\newcommand{\cB}{\mathcal{B}}

\theoremstyle{remark}

\usepackage{geometry}
\title{On the order of  4-dimensional regular polytope numbers}
\author{Anji Dong, The Nguyen, Alexandru Zaharescu}

\address{
Anji Dong: Department of Mathematics,
University of Illinois Urbana-Champaign,
Altgeld Hall, 1409 W. Green Street,
Urbana, IL, 61801, USA}
\email{anjid2@illinois.edu}

\address{
The Nguyen: Department of Mathematics,
University of Illinois Urbana-Champaign,
Altgeld Hall, 1409 W. Green Street,
Urbana, IL, 61801, USA}
\email{thevn2@illinois.edu}

\address{
Alexandru Zaharescu: Department of Mathematics,
University of Illinois Urbana-Champaign,
Altgeld Hall, 1409 W. Green Street,
Urbana, IL, 61801, USA and Simion Stoilow Institute of Mathematics of the Romanian Academy, 
P. O. Box 1-764, RO-014700 Bucharest, Romania}
\email{zaharesc@illinois.edu}

\begin{document}

\setcounter{tocdepth}{1}
\begin{abstract}
    In light of Kim's conjecture on regular polytopes of dimension four, which is a generalization of Waring's problem, we establish asymptotic formulas for representing any sufficiently large integer as a sum of numbers in the form of those regular 4-polytopes. Moreover, we are able to obtain a more general result of the asymptotics for any degree-four polynomial $f$ satisfying $f(0)=0$ and $f(1)=1$.
\end{abstract}
\maketitle

\section{Introduction}
In 1770, Lagrange showed that every non-negative integer can be written as a sum of four squares. In the same year, Waring generalized the question to higher degrees, known as Waring's problem, which asks whether for each $k\in\N$, there exists a natural number $s$ such that every positive integer $m$ is the sum of at most $s$ natural numbers raised to the power $k$. In 1909, Hilbert \cite{hilbert1909} showed that such $s$ exists for every $k\in\N$. Define $g(k)$ to be the least such number $s$ having the above property. Wieferich \cite{Wieferich1908BeweisDS} and Kempner \cite{Kempner1912BemerkungenZW} showed that $g(3)=9$. Balasubramanian \cite{balasubramanian}, Deshouillers and Dress \cite{deshouillers} proved that $g(4)=19$. Chen \cite{chen1964} confirmed that $g(5)=37$ and Pillai \cite{pillai1940} proved that $g(6)=73$. Along this direction, results are generalized to regular polytopes of dimension $k$. We first introduce a theorem on classification of regular polytopes, see Coxeter \cite{coxeter}. 
\begin{thm}[Schl\"afli]
    The only possible Schl\"afli symbols for a regular polytope in the Euclidean space in $\R^d$ are given by the following list:
    \begin{enumerate}
        \item[] $d=2: \{n\}$, where $n$ is an arbitrary integer;
        \item[] $d=3: \{3,3\},\{3,4\},\{4,3\},\{3,5\},\{5,3\}$
        \item[] $d=4: \{3,3,3\},\{3,3,4\},\{4,3,3\},\{3,4,3\},\{3,3,5\},\{5,3,3\}$
        \item[] $d\geq 5: \{3^{d-1}\},\{3^{d-2},4\},\{4,3^{d-2}\}.$
    \end{enumerate}
    Note that for any $d\geq 3$, a $d$-th power is $\{4,3^{d-2}\}$.
\end{thm}

Let $\beta_f$ denote the Schl\"afli symbol corresponds to a regular polytope represented by the polynomial $f$. We extend the definition of $g(k)$ by defining $g(\beta_f)$ to be the least integer $s$ such that every integer is a sum of at most $s$ numbers represented by $f$. In the 1850s, Sir Frederick Pollock \cite{pollock} made conjectures on regular polytopes of dimension three. More precisely, in the language of Schl\"afli symbols, he conjectured that $g(\{3,3\})=5, g(\{3,4\})=7, g(\{4,3\})=9, g(\{3,5\})=13$, and $g(\{5,3\})=21$. The cube case $g(\{4,3\})=9$ was mentioned above, and more recently, Basak, Saettone, and two of the authors \cite{dong2024pollock} corrected and proved the conjectures of   and dodecahedral numbers, which correpsond to  $g(\{3,5\})$ and $g(\{5,3\})$. They showed that $g(\{3,5\})=15$ and $g(\{5,3\})=22$. For regular polytopes of higher dimensions, through numerical data, Kim \cite{kim2002} made conjectures on dimension four to seven. In particular, for dimension four, he conjectured that $g(\{3,3,3\})=8,g(\{3,3,4\})=11,g(\{4,3,3\})=19,g(\{3,4,3\})=28,g(\{3,3,5\})=125$, and $g(\{5,3,3\})=606$.

Yet another direction, people consider the minimal integer $s$ such that a sufficiently large integer $m$ can be written as at most $s$ $k$-th powers of natural numbers. Denote such number $s$ by $G(k)$, called the order of $k$-th powers. Linnik \cite{linnik1943} showed that $G(3)\leq 7$. Davenport \cite{davenport1939} proved $G(4)=16$, Vaughan and Wooley \cite{vaughanwooley1994, vaughanwooley1995} established that $G(5)\leq 17$ and $G(6)\leq 24$. More recently, Wooley \cite{wooley2016} confirmed that $G(7)\leq 31, G(8)\leq 39$, and $G(9)\leq 47$. In 2023, Br\"udern and Wooley \cite{brudernwooley2022} showed that for all natural numbers $k$, $G(k)\leq \lceil k(\log k+4.20032)\rceil$. For any $k,s\in\N$, let $\mathcal{R}_{k,s}(m)$ denote the number of ways of writing $m$ as a sum of $s$ $k$-th powers. Hardy, Littlewood, and Ramanujan \cite{ littlewood1920, hardy1918} showed that for $k\geq 3$ and $s> 2^k$ and for sufficiently large $m$,
\begin{align}\label{asymptotics of m as kth powers}
\mathcal{R}_{k,s}(m)\asymp \frac{\Gamma(1+1/k)^s}{\Gamma(s/k)}\mathfrak{S}_{k,s}(m) m^{s/k-1},
\end{align}
where $\mathfrak{S}_{k,s}(m)>0$.

We generalize the notion of $G(k)$ by defining $G(\beta_f)$ to be the least integer $s$ such that every sufficiently large integer is a sum of at most $s$ polytope numbers represented by $f$. For three-dimensional polytopes, Brady \cite{Brady_2015} proved that all integers larger than $e^{10^7}$ can be written as a sum of at most seven numbers represented by the Schl\"afli symbol $\{3,4\}$, so $G(\{3,4\})\leq 7$. More recently, Basak, Saettone, and two of the authors \cite{dong2024pollock} showed that $G(\{3,5\})\leq 9$ and $G(\{5,3\})\leq 9$.

In this paper, we are interested in regular polytopes of dimension four, in particular, those three with Schl\"afli symbols $\{3,4,3\}, \{3,3,5\}$ and $\{5,3,3\}$. To be precise, for any $n\in\N$, the corresponding $n$-th polytope numbers are in order defined as
\begin{align*}
 f_1(n)= n^2(3n^2-4n+2), \quad f_2(n) = \frac{n}{6}(145n^3-280n^2+179n-38),   
\end{align*}
and
\begin{align}\label{target polynomials}
    f_3(n) = \frac{n}{2}(261n^3-504n^2+283n-38).
\end{align}

For $i=1,2,3$, denote the set of numbers represented by $f_i$ as $\mathcal{F}_i$, i.e., 
\[
\mathcal{F}_i = \{f_i(n): n\in\N\cup \{0\}\}.
\]
Our goal is to obtain asymptotic formulas, similar to \eqref{asymptotics of m as kth powers}, by employing the Hardy--Littlewood method for the number of ways a sufficiently large positive integer $m$ can be written as a sum of numbers from $\mathcal{F}_i$, with explicit power-saving error terms. Our main result is as follows. \begin{thm}\label{thm: representations}
    
For $m \in \mathbb{N}$ and $i=1,2,3$, let $\mathcal{R}_{f_i,17}(m)$ denote the number of ways to represent $m$ as the sum of $17$ numbers from the set $\mathcal{F}_i$. Then for any $m>e^{e^{471}}$, we have 
\begin{align}
\bigg|\mathcal{R}_{f_i,17}(m) -\bigg(\frac{24}{A_i}\bigg)^{17/4}\mathfrak{S}_{f_i,17}(m)\Gamma\bigg(\frac{5}{4}\bigg)^{17}\Gamma\bigg(\frac{17}{4}\bigg)^{-1}m^{13/4}\bigg| \leq \num{1.6e42}\cdot m^{\frac{13}{4}-\frac{7}{1488}},
\end{align}
where $A_1 = 72, A_2=580, A_3=3132$, $\Gamma$ is the Gamma function and all $\mathfrak{S}_{f_i,17}(m)$ uniformly satisfy the inequality \eqref{Sigma Bound} with $s=17$ below.
\end{thm}
In Section \ref{sec: Proof of Main Theorem}, we establish a more general result by considering sums of $s$ polytope numbers from each $\mathcal{F}_i$ for any $s\geq 17$. 

An immediate corollary on the order of $f_i$ for $i=1,2,3$ is the following.
\begin{coro}
 $G(\{3,4,3\}), G(\{3,3,5\})$, and $G(\{5,3,3\})$ are at most $17$.
\end{coro}
In fact, we can establish a general result for any degree-four polynomial $f$ satisfying $f(0)=0$ and $f(1)=1$. It's obvious that any such polynomial can be represented by 
\begin{align} f(n) = A \binom{n}{4} + B \binom{n}{3} + C \binom{n}{2} + n\label{defn: f(n)}
\end{align}
for some $A,B,C \in \Z$ such that $A\geq  1$ and $n\in\N$. We will assume that the coefficients are not too large. More precisely, the absolute values of the coefficients are less than $e^{e^{10}}$. Although this assumption is not strictly necessary, we use it in order to obtain more effective bounds for Theorem \ref{thm: representations}. The result is as follows.
\begin{thm}\label{thm: representations general for deg 4}
    For $m \in \mathbb{N}$ and a degree-four polynomial $f$ represented by \eqref{defn: f(n)}, with $A,|B|,|C|\leq e^{e^{10}}$, let $\mathcal{R}_{f,17}(m)$ denote the number of ways to write $m$ as the sum of $17$ numbers represented by $f$. Then, there exists an explicitly computable integer $m_f$ such that for any $m\geq m_f$, we have 
    \begin{align}
\bigg|\mathcal{R}_{f,17}(m) -\bigg(\frac{24}{A}\bigg)^{17/4}\mathfrak{S}_{f,17}(m)\Gamma\bigg(\frac{5}{4}\bigg)^{17}\Gamma\bigg(\frac{17}{4}\bigg)^{-1}m^{13/4}\bigg| \leq \num{1.6e42}\cdot m^{\frac{13}{4}-\frac{7}{1488}},
\end{align}
where $\Gamma$ is the Gamma function and $\mathfrak{S}_{f,17}(m)>0$.
\end{thm}
Thus, Theorem \ref{thm: representations} is a special case of Theorem \ref{thm: representations general for deg 4}. In Section \ref{sec: Proof of Main Theorem}, we establish a more general result by considering sums of $s$ polytope numbers represented by $f$ for any $s\geq 17$. We also give explicit bounds for $m_f$ and $\mathfrak{S}_{f,17}(m)$ with respect to the degree-four polynomial $f$. 

\subsection*{Structure of the Paper} The paper is organized as follows. Section \ref{sec: initial set-up} covers the initial setup needed for the Hardy--Littlewood circle method, which is used to obtain the asymptotic formulas for the number of representations using numbers represented by $f$, defined in Theorem \ref{thm: representations general for deg 4}. Sections \ref{sec: minor arcs} contains details of the minor arc estimates. Sections \ref{sec: major arcs}, \ref{sec: Major Arcs II} and \ref{sec: Major Arcs Sec III} are devoted to the major arc estimates. Lastly, we present the proofs of Theorem \ref{thm: representations} and Theorem \ref{thm: representations general for deg 4} in Section \ref{sec: Proof of Main Theorem}. 

\subsection*{General Notation} We employ some standard notation that will be used throughout the article.
\begin{itemize}
    \item Throughout the paper, the expressions $f(X)=O(g(X))$, $f(X) \ll g(X)$, and $g(X) \gg f(X)$ are equivalent to the statement that $|f(X)| \leq C|g(X)|$ for all sufficiently large $X$, where $C>0$ is an absolute constant. 
    \item We define by $d(n)$, the divisor function defined on $\mathbb{N}$ by $d(n) = \sum_{d \mid n}1.$
    \item We write $e(\theta)$ to denote the expression $\exp{(2\pi i \theta)}$.
    \item Given $\alpha \in \mathbb{R}$, the notation $\|\alpha\|$ denotes the smallest distance of $\alpha$ to an integer.
    \item The letter $p$ always denotes a prime number.
\end{itemize}
 
\section{Setup}\label{sec: initial set-up}
In this section, we outline the initial setup required to employ the Hardy--Littlewood circle method to prove Theorem \ref{thm: representations general for deg 4}. Let $f$ be a polynomial satisfying \eqref{defn: f(n)}. For convenience, we denote by $f_n$ the value of $f(n)$ for $n \in \N$. 

For a sufficiently large positive integer $m$, let $\mathcal{R}_{f, s}(m)$ denote the number of ways to write $m$ as a sum of $s$ numbers of form $f(n)$ for some $n \in \N$, that is,
\begin{equation}\label{eq:def-f_s}
   \mathcal{R}_{f, s}(m) = \left|\left\{ (n_1, \dots, n_s) \in \N^s \,:\, m = f(n_1) + \dots + f(n_s)\right\}\right|. 
\end{equation}
Let $\alpha \in \R$ and $N \in \N$ such that $m\leq f_N$. Then, define 
\begin{equation} \label{eq:def-S(alpha)}
S_{f,N}(\alpha) := \sum_{i = 1}^N e(\alpha f_n).  
\end{equation}
For convenience, we write $S_f(\alpha) = S_{f,N}(\alpha)$. It follows that
\begin{align*}
    S_f(\alpha)^s = \sum_{n_1}^N \sum_{n_2=1}^N \dots \sum_{n_s=1}^N e(\alpha(f_{n_1} + f_{n_2} + \dots + f_{n_2})) = \sum_{n = 1}^{ s f_N} \mathcal{R}_{f, s}(n) e(\alpha n).
\end{align*}
Applying Cauchy's formula, we obtain 
\begin{equation}\label{eq::f_s(m)-formula}
  \mathcal{R}_{f, s}(m) = \int_{0}^1 S_f(\alpha)^s e(-\alpha m) \dd\alpha.  
\end{equation}
Next, we choose
\begin{align}\label{Defining N}
   N = \bigg \lceil \bigg (\frac{24m}{A}\bigg )^{\frac{1}{4}} \bigg \rceil +1,
\end{align}
and let $ P=N^\delta$ for some $\delta\in\R$ such that $N^{3\delta -4} < 1/2$. We then define $\mathfrak{I}=\left(N^{\delta-4}, 1+N^{\delta-4}\right]$. For integers $1 \leq a \leq q \leq P$ and $(a,q) = 1$, we define the major arcs $\mathfrak{M}(q, a)$ as  
\begin{align*}
    \fM(q,a) = \{ \alpha \,:\, |\alpha - a/q| \leq N^{\delta-4}\}.
\end{align*}
Let $\mathfrak{M}$ denote the union of the $\mathfrak{M}(q, a)$'s, i.e.,
\[
\mathfrak{M}=\bigcup_{\substack{1\leq a\leq q\leq P\\(a,q)=1}} \mathfrak{M}(q,a).
\]
Observe that $\mathfrak{M} \subset \mathfrak{I}$, so we define the minor arcs as 
\[
\mathfrak{m}=\mathfrak{I} \backslash \mathfrak{M}.
\]
For distinct pairs $a / q \neq a^{\prime} / q^{\prime}$ with $1 \leq q, q^{\prime} \leq P$, we have
$$
\left|\frac{a}{q}-\frac{a^{\prime}}{q^{\prime}}\right| \geq  \frac{1}{q q^{\prime}}>N^{-2\delta}>  2 N^{\delta-4}.
$$
Thus, the $\mathfrak{M}(q, a)$'s are pairwise disjoint. Finally, using \eqref{eq::f_s(m)-formula}, we split $\mathcal{R}_{f, s}(m)$ into integrals over the major and minor arcs:
\begin{align}\label{eq::Major+Minor Arcs}
\mathcal{R}_{f, s}(m) = \int_{\mathfrak{M}} S_f(\alpha)^s e(-\alpha m) \dd \alpha+\int_{\mathfrak{m}} S_f(\alpha)^s e(-\alpha m) \dd \alpha.
\end{align}

\section{Minor Arcs}\label{sec: minor arcs}

In this section, we estimate the integral over the minor arcs. More precisely, we prove the following theorem, which shows that the integral over the minor arcs in \eqref{eq::Major+Minor Arcs} is \textit{small}. 
\begin{thm}\label{thm::minor-arc}
Let $S_f(\alpha)$ be defined in \eqref{eq:def-S(alpha)}. Then for $s \geq 17$ and $N \geq e^3$, we have 
\begin{align*}
      \left|\int_{\fm} S_f(\alpha)^s e(-m \alpha) \dd\alpha \right| \leq 10^6 \cdot 11^{s-16} A^{\frac{s-16}{8}}(\log N)^{\frac{s-16}{8}} N^{s-4 - \frac{\delta(s-16)}{8} + \frac{s}{\log\log N}}.
\end{align*}
\end{thm}

In order to prove Theorem \ref{thm::minor-arc}, we first establish some preliminary lemmas that are necessary in our treatment. 

\begin{lem}[{\cite[Theorem 1]{nicolas1983majorations}}] \label{lem::Explicit Divisor Bound}
Let $d(n)$ denote the divisor function. Then for $n \geq  e^3$, we have
\[ d(n) \leq n^{\frac{1.0661}{\log \log n}}.
\]
\end{lem}

Let $\psi(x)$ be a real-valued function of $x$. For $j\in\N\cup\{0\}$, define the $j$-th forward difference operator as  
\begin{align*}
\Delta_0(\psi(x) ; h)&=\psi(x)\\
\Delta_1(\psi(x) ; h)&=\psi(x+h)-\psi(x),
\end{align*}
and 
$$
\begin{aligned}
\Delta_j(\psi(x) ; \mathbf{h}) =\Delta_j\left(\psi(x) ; h_1, \ldots, h_j\right) =\Delta_1\left(\Delta_{j-1}\left(\psi(x) ; h_1, \ldots, h_{j-1}\right) ; h_j\right) .
\end{aligned}
$$
It's easy to see that when $1 \leq j \leq k$,
$$
\Delta_j\left(x^k ; \mathbf{h}\right)=h_1 \ldots h_j p_j\left(x ; h_1, \ldots, h_j\right),
$$
where $p_j$ is a polynomial in $x$ of degree $k-j$ with leading coefficient $k ! /(k-j)$ !. By the linearity of the operator $\Delta_j$, it follows that
\[
\Delta_j\left(a_k x^k+\ldots+a_1 x ; \mathbf{h}\right)=\sum_{i=1}^k a_i \Delta_j\left(x^i ; \mathbf{h}\right).
\]

\begin{lem}[{\cite[Lemma 2.3]{vaughan1997hardy}}] \label{lem::Weyl Differencing}
Let $\psi(x)$ be a real-valued arithmetic function, and suppose
$$
F(\psi)=\sum_{1 \leq x \leq X} e(\psi(x)).
$$
Then for each $j \in \mathbb{N}$,
$$
|F(\psi)|^{2^j} \leq(2 X)^{2^j-j-1}  \sum_{\left|h_1\right|<X} \cdots \sum_{\left|h_j\right|<X} \sum_{x \in T_j(\mathbf{h})} e\left(\Delta_j(\psi(x) ; \mathbf{h})\right) ,
$$
where $T_j(\mathbf{h})$ denotes the interval of integers defined by putting $T_0(h)=[1, X]$, and for $j \geq  1$, we recursively set
$$
T_j\left(h_1, \ldots, h_j\right)=T_{j-1}\left(h_1, \ldots, h_{j-1}\right) \cap\left\{x \in[1, X]: x+h_j \in T_{j-1}\left(h_1, \ldots, h_{j-1}\right)\right\}.
$$
\end{lem}
We also make use of the following lemmas, which can be found in \cite{dong2024pollock}.

\begin{lem}[{\cite[Lemma 3.3]{dong2024pollock}}]\label{lem::Geometric Sum}
Let $\alpha \in \mathbb{R}$. Let $X$ and $Y$ be real numbers with $Y>1$. Then
$$
\left \lvert \sum_{X<x \leq X+Y} e(\alpha x) \right \rvert \leq \min \left\{Y+1,\frac{1}{2}\|\alpha\|^{-1}\right\}.
$$
\end{lem}

\begin{lem}[{\cite[Corollary 3.6]{dong2024pollock}}]\label{lem::Weyl's Inequality with Eta}
Let $X,Y \in \mathbb{R}$ with $X,Y \geq  1$. Let $\alpha,\beta \in \mathbb{R}$ and suppose there exist $a \in \mathbb{Z}$ and $q \in \mathbb{N}$ with $(a, q)=1$, $q>100$ and $|\alpha-a / q| \leq \eta q^{-2}$, for some absolute constant $\eta \geq  1$. Then
$$
\sum_{1 \leq x \leq X} \min \left\{Y,\|\alpha x+\beta\|^{-1}\right\} \leq 8X Y\eta\left(q^{-1}+Y^{-1}+X^{-1}+q(X Y)^{-1}\right) \log q .
$$
\end{lem}

For $\balp = (\alpha_1, \dots, \alpha_4)\in\R^4$ and $\psi(x) = \alpha_1 x + \dots +\alpha_4 x^4$, let 
\begin{equation*}
F(\balp) = \sum_{1 \leq x \leq X} e(\psi(x))  .
\end{equation*}

\begin{lem}\label{lem::upper-F(a)}
Let $\eta \geq 1$ be fixed and $X \geq e^{3}$. Let $\balp = (\alpha_1, \dots, \alpha_4) \in \R^4$. Suppose there exist $a \in \Z$ and $q \in \N$ which satisfy $q > 100$ and $|\alpha_4 - a/q| \leq \eta q^{-2}$. Then 
\begin{equation}\label{eq:bound.F(alpha)}
    |F(\balp)| \leq 2X^{7/8} + 5 \eta^{1/8} X^{1 + \frac{3.1983}{4\log(3\log X)}} \left( q^{-1} + X^{-1} + q X^{-4}\right)^{1/8} (\log q)^{1/8}.
\end{equation}
\end{lem}

\begin{proof}
Clearly, the desired estimate is trivial when $q \geq X^4$, so we may assume $q < X^{4}$. Applying Lemma \ref{lem::Weyl Differencing} with $j = 3$, we obtain 
\begin{align*}
    |F(\balp)|^8 \leq (2X)^4\sum_{|h_1| < X} \sum_{|h_2| < X} \sum_{|h_3| < X} \cE(\bh),
\end{align*}
where 
\[ \cE(\bh) = \sum_{x \in T_3(\bh)} e(\Delta_3(\psi(x);\bh)\]
and $T_3(\bh)$ is a suitable interval of integers contained in $[1,X]$. Note that $\Delta_3(\psi(x);\bh) = 24h_1h_2h_3 x \alpha_4 +  r$, where $r = r(\alpha,\bh)$ is independent of $x$. It follows from Lemma \ref{lem::Geometric Sum} that 
\begin{equation*}
    \cE(\bh) \leq \min\{X + 1, \Vert 24h_1h_2h_3 \alpha_4 \Vert^{-1}\}.
\end{equation*}
Thus,
\begin{align*}
  |F(\balp)|^8 &\leq (2X)^4 \left( \sum_{|h_1| < X} \sum_{|h_2| < X} \sum_{|h_3| < X} \min\{ X + 1, \Vert 24 h_1h_2h_3 \alpha_4 \Vert^{-1} \right)  \\
  &\leq 16 X^4 \left( 12X^3 + 8 \sum_{1 \leq n \leq 24X^3} \left(d\left(\frac{n}{24}\right)\right)^2 \min\{X+1, \Vert n\alpha_4 \Vert^{-1}  \right),
\end{align*}
where the last inequality is obtained by accounting for the summands where $h_1h_2h_3 = 0$. We now use Lemma \ref{lem::Explicit Divisor Bound} and Lemma \ref{lem::Weyl's Inequality with Eta} to obtain the bound 
\begin{align*}
    |F(\balp)|^{8} &\leq 192 X^7 + 2^7 X^{4 + \frac{6.3966}{\log(3\log X)}} \sum_{1 \leq n \leq 24 X^3} \min\{X + 1, \Vert n \alpha_4 \Vert^{-1}\} \\ 
    &\leq 192 X^7 + 10^5 X^{8 + \frac{6.3966}{\log(3\log X)}} \eta \left(\frac{1}{q} + \frac{1}{X} + \frac{q}{X^4}\right)\log q.
\end{align*}
This implies that 
\begin{align*}
    |F(\balp)| \leq 2X^{7/8} + 5 \eta^{1/8} X^{1 + \frac{3.1983}{4\log(3\log X)}} \left( q^{-1} + X^{-1} + q X^{-4}\right)^{1/8} (\log q)^{1/8},
\end{align*}
which completes the proof of Lemma \ref{lem::upper-F(a)}.
\end{proof}

\begin{lem}\label{lem:upper-int.S(alpha)}
Let $S_f(\alpha)$ be as defined in \eqref{eq:def-S(alpha)}. Suppose $1 \leq j \leq 4$. Then for $N \geq e^3$, 
\begin{equation*}
    \int_{0}^1 |S_f(\alpha)|^{2^{j}}\dd\alpha \leq 10^6 N^{2^j - j + \frac{12.7932}{\log\log N}}.
\end{equation*}
\end{lem}
\begin{proof}
The proof for $j=1,2,3$ is analogous to the proof of Lemma 4.2 in \cite{dong2024pollock}, with the only difference being that we need to count the number of solutions for a polynomial of degree $3$ instead of a quadratic polynomial. Therefore, we will omit the proof and only state the results for the cases $j = 1, 2, 3$. Specifically, we have
\begin{align*}
    \int_{0}^1 |S_f(\alpha)|^2 \dd \alpha &\leq N, \\ 
    \int_0^1 |S_f(\alpha)|^4 \dd \alpha &\leq 13 N^{2 + \frac{4.2644}{\log\log N}}, \\
    \int_0^1 |S_f(\alpha)|^8 \dd \alpha &\leq 328 N^{5 + \frac{8.5288}{\log\log N}}.
\end{align*}
Next, we consider the case when $j = 4$.  It follows from Lemma \ref{lem::Weyl Differencing} that 
\begin{align*}
   |S_f(\alpha)|^8 \leq (2N)^{4} \sum_{|h_1| < N} \sum_{|h_2| < N} \sum_{|h_3| < N} \sum_{n \in T_3(\bh)} e(\alpha \cdot \Delta_3(f_n; \bh)).
\end{align*}
Therefore,
\begin{align*}
  \int_{0}^1 |S_f(\alpha)|^{2^4} \dd\alpha = \int_0^{1} S_f(\alpha)^4 S_f(-\alpha)^4 |S_f(\alpha)|^8 \dd \alpha \leq (2N)^4 \cdot W,  
\end{align*}
where 
\begin{equation*}
    W = \sum_{|h_1| < N} \sum_{|h_2| < N} \sum_{|h_3| < N} \sum_{n \in T_3(\bh)} \int_0^1  S_f(\alpha)^4 S_f(-\alpha)^4 e(\alpha \cdot \Delta_3(f_n; \bh))\dd\alpha,
\end{equation*}
and $T_3(\bh)$ is a suitable subinterval of $[1,N]$. By orthogonality, the expression $W$ is bounded above by the number of integral solutions to the equation 
\begin{align*}
    \sum_{i = 1}^4 (f_{u_i} - f_{v_i}) = \Delta_3(f_n;\bh),
\end{align*}
where $1 \leq u_i, v_i \leq N$ for all $1 \leq i \leq 4$, $1 \leq n \leq N$ and $|h_j| < N$ for all $1 \leq j \leq 3$. We now bound $W$ by dividing into two cases: $\Delta_3(f_n; h) = 0$ and $\Delta_3(f_n; h) \neq 0$. 

The first case is $\sum_{i = 1}^4 (f_{u_i} - f_{v_i}) = 0$, which implies $\Delta_1(f_n;h) = 0$. By orthogonality, the number of choices for $u_i$ and $v_i$ ($1 \leq i \leq 4$) in this case is 
\[ \int_0^1 S_f(\alpha)^4 S_f(-\alpha)^4 \dd \alpha = \int_0^1 |S_f(\alpha)|^8 \dd\alpha \leq 328 N^{5 + \frac{8.5288}{\log\log N}}.\]
On the other hand, we have $\Delta_3(f_n;\bh) = h_1h_2h_3 L(n;h)$, where $L(n;h)$ is a linear polynomial in $n$, determined by the choices of $h_1, h_2$ and $h_3$. So, either $h_1h_2h_3 = 0$, or $n$ is a zero of $L$. Therefore, the total number of choices for $n$ and $\bh$ is at most $20N^3$. Hence, the contribution in this case to $W$ is bounded above by 
\[
6560 N^{8 + \frac{1.0661}{\log\log N}}.
\]
For the second case, we write $\sum_{i = 1}^4 (f_{u_i} - f_{v_i}) = k$ for some non-zero integer $k$ with $|k| \leq 4 f_N$. For each such choice of $u_i$ and $v_i$ ($i = 1, \dots, 4$), we have $h_1h_2h_3 L(n;\bh) = k$ and thus there are at most $8 d(k)^3$ choices for $\bh$. For each such choice of $\bh$, there is at most one choice for $n$. This shows that the total choices for $n$ and $\bh$ is bounded above by 
\[ 8 d(k)^3 \leq 8 N^{\frac{12.7932}{\log\log N}}.\]
Since there are at most $N^8$ choices for $u_i$ and $v_i$ with $1 \leq i \leq 4$, the contribution to $W$ in this case is at most 
\[
8N^{8 + \frac{12.7932}{\log\log N}}.
\]
Thus, we obtain 
\begin{align*}
\int_0^1 |S_f(\alpha)|^{16} \dd \alpha \leq (2N)^4 W \leq 16N^4 \left(6560 N^{8 + \frac{8.5288}{\log\log N}} + 8 N^{8 + \frac{12.7932}{\log\log N}}\right) \leq  10^6 N^{12 + \frac{12.7932}{\log\log N}},  
\end{align*}
as desired. 
\end{proof}

We are now ready to prove Theorem \ref{thm::minor-arc}. 
\begin{proof}[Proof of Theorem \ref{thm::minor-arc}]
We can write 
\begin{equation*}
    \left| \int_{\fm} S_f(\alpha)^s e(-m \alpha) \dd\alpha\right| \leq \left(\sup_{\alpha \in \fm} |S_f(\alpha)|\right)^{s-16} \int_{0}^1 |S_f(\alpha)|^{16} \dd\alpha.  
\end{equation*}
Consider an arbitrary point $\alpha$ of $\mathfrak{m}$. By Dirichlet's Theorem (see \cite[Lemma 2.1]{vaughan1997hardy}), there exist $a, q$ with $(a, q)=1$ and $q \leq N^{3-\delta}$ such that $|\alpha-a / q| \leq q^{-1} N^{\delta-3}$. Since $\alpha \in \mathfrak{m} \subset\left(N^{\delta-3}, 1-N^{\delta-3}\right)$ it follows that $1 \leq a \leq q$. Therefore, $q>N^\delta$, for otherwise $\alpha$ would lie in $\mathfrak{M}$. We now apply Lemma \ref{lem::upper-F(a)} with $\eta = 24A$. Then we have
\begin{align*}
|S_f(\alpha)| &\leq 2 N^{7/8} + 5 \cdot (24A)^{1/8} N^{1 + \frac{3.1983}{4\log\log N}} \left(q^{-1} + N^{-1} + qN^{-3}\right)^{1/8} (\log N)^{1/8} \\ 
    &\leq 2 N^{7/8}  + 5 \cdot (72A)^{1/8} N^{1 - \frac{\delta}{8} + \frac{3.1983}{4\log\log N}} (\log N)^{1/8} \\ 
    &\leq 11A^{1/8} N^{1 - \frac{\delta}{8} + \frac{3.1983}{4\log\log N}} (\log N)^{1/8}.
\end{align*}
Additionally, by Lemma \ref{lem:upper-int.S(alpha)}, we have 
\begin{equation*}
    \int_{0}^1 |S_f(\alpha)|^{16} \dd \alpha \leq 10^6 N^{12 + \frac{12.7932}{\log\log N}}.
\end{equation*}
It follows that 
\begin{align*}
    \left|\int_{\fm} S_f(\alpha)^s e(-m \alpha) \dd\alpha \right| &\leq 10^6 \left(11A^{1/8} N^{1-\frac{\delta}{8} + \frac{3.1983}{4\log\log N}}\right)^{s-16} (\log N)^{\frac{s-16}{8}} N^{12 + \frac{12.7932}{\log\log N}} \\ 
    &\leq 10^6 \cdot 11^{s-16} A^{\frac{s-16}{8}}(\log N)^{\frac{s-16}{8}} N^{s-4 - \frac{\delta(s-16)}{8} + \frac{s}{\log\log N}},
\end{align*}
which completes the proof of Theorem \ref{thm::minor-arc}.
\end{proof}

\section{Major Arcs}\label{sec: major arcs}
In this section, our primary objective is to effectively approximate $S_f(\alpha)$ for $\alpha \in \mathfrak{M}$. Let $\theta = \alpha-a/q$. For convenience, we extend the definition in \eqref{defn: f(n)} by writing $f_t$ for $t \in \mathbb{R}$. More precisely, we define
\begin{align*}
    f_t &= A \binom{t}{4} + B \binom{t}{3} + C \binom{t}{2} + t\\
    &= \frac{At^4}{24}-\frac{At^3}{4}+\frac{11At^2}{24}-\frac{At}{4}+\frac{Bt^3}{6}-\frac{Bt^2}{2}+\frac{Bt}{3}+\frac{Ct^2}{2}-\frac{Ct}{2}+t\\
    \textrm{and} \quad\quad\quad&\\
    f_t' &= \frac{A(2t^3-9t^2+11t-3)}{12}+\frac{B(3t^2-6t+2)}{6}+\frac{C(2t-1)}{2}+1,
\end{align*}
for $t \in \mathbb{R}.$ Define 
\begin{align}
M(t) &:=  \sum_{1\leq n\leq t} e\bigg (\frac{a}{q}f_n \bigg), \label{Definition M(t)} \\
\textrm{and} \quad V(q,a) &:= \sum_{1\leq n\leq 24q} e\bigg (\frac{a}{q}f_n \bigg) \label{Definition V(q,a)}.
\end{align} 
To this end, we aim to prove the following lemma. 
\begin{lem}\label{Approx 4}
Let $N \geq  3A+2|B|$. Then
\[
\bigg \lvert \int_{\mathfrak{M}}S_f(\alpha)^s e(-\alpha m) \dd\alpha -\int_{\mathfrak{M}} \bigg (\frac{V(q,a)}{24q}\int_1^N e\bigg( \frac{At^4 \theta}{24} \bigg) \dd t \bigg)^s  e(-\alpha m) \dd\alpha \bigg \rvert \leq 70(A+4|B|)\pi s N^{5\delta+s-5}.
\]
\end{lem}

To begin with, since $\mathfrak{M}(q, a)$ are pairwise disjoint, we have
\begin{align}\label{Major Arc Disjointness}
\int_{\mathfrak{M}} S_f(\alpha)^s e(-\alpha m)d\alpha= \sum_{q\leq N^\delta}\sum\limits_{\substack{a=1 \\ (a,q)=1}}^q\int_{\mathfrak{M}(q,a)} S_f(\alpha)^se(-\alpha m) \dd\alpha.
\end{align}

Applying partial summation, we obtain the following lemma.
\begin{lem}\label{Partial Summation}
Let $\alpha \in \mathfrak{M}(q,a)$ and $\theta = \alpha - a/q$. Then 
\begin{align}
 S_f(\alpha) &=M(N)e(\theta f_N)-2\pi i\theta\int_{1}^N M(t) f_t'e(\theta f_t) \dd t.
\end{align}
\end{lem}

\begin{lem}\label{Congruence}
For any $q$, $f_n\equiv f_{n+24q}$ mod $q$.
\end{lem}
\begin{proof}
    The proof follows easily by noticing that the least common multiple of the denominators for $f_n$ is $24$.
\end{proof}

\begin{lem}\label{Approx 1}
For all $1 \leq t \leq N$, with $M(t)$ defined as in \eqref{Definition M(t)}, we have
\begin{align*}
   \bigg \lvert M(t) - \frac{V(q,a)}{24q} t \bigg \rvert \leq 24q.
\end{align*}
\end{lem}
\begin{proof}
    The arguments follows similarly to Lemma 5.3 in \cite{dong2024pollock}, with $2q$ replaced by $24q$.
\end{proof}
\begin{lem}\label{Approx 2}
Let $\alpha \in \mathfrak{M}(q,a)$ and $\theta = \alpha - a/q$.
Then
\[
\bigg \lvert S_f(\alpha)-\frac{V(q,a)}{24q}\int_1^N e(\theta f_t)\dd t \bigg \rvert \leq 24q+1+(2A+8|B|)q\pi\theta N^4.
\]
\end{lem}
\begin{proof}
By applying Lemma \ref{Approx 1}, we obtain the bounds
\begin{align*}
    \bigg \lvert M(N)e(\theta f_N)-\frac{V(q,a)}{24q} N e(\theta f_N)\bigg \rvert 
   &\leq 24q,\\
\bigg \lvert 2\pi i\theta \int_{1}^NA(t)f_t'e(\theta f_t)\dd t-2\pi i\theta\int_{1}^N\frac{V(q,a)}{24q}tf_t'e(\theta f_t)\dd t \bigg \rvert &\leq 48\pi\theta\int_{1}^N q f_t'\dd t \leq (2A+8|B|)q\pi\theta N^4.
\end{align*}
Thus, by Lemma \ref{Partial Summation} and the triangle inequality, we have
\begin{align} \label{Approx 2 Step 1}
\bigg \lvert S_f(\alpha)-\frac{V(q,a)}{24q}Ne(\theta f_N)+2\pi i\theta\int_{1}^N&\frac{V(q,a)}{24q}t f_t'e(\theta f_t)\dd t \bigg \rvert \leq 24q+(2A+8|B|)q\pi\theta N^4.
\end{align}
Applying integration by parts, we find
\begin{align}
\frac{V(q,a)}{24q} &N e(\theta f_N)-2\pi i\theta\int_{1}^N\frac{V(q,a)}{24q}tf_t'e(\theta f_t) \dd t \notag \\
&=\frac{V(q,a)}{24q}e(\theta)+\frac{V(q,a)}{24q}\int_1^N e(\theta f_t) \dd t. \label{Approx 2 Step 2}
\end{align}
Finally, combining \eqref{Approx 2 Step 1} and \eqref{Approx 2 Step 2}, and trivially bounding $V(q,a)$ complete the proof.
\end{proof}
\begin{lem}\label{Approx 3}
Let $N \geq  6A+4|B|, \alpha \in \mathfrak{M}(q,a)$ and $\theta = \alpha - a/q$. Then
\[
\bigg \lvert S_f(\alpha)-\frac{V(q,a)}{24q}\int_1^N e\bigg( \frac{At^4 \theta}{24} \bigg) \dd t \bigg \rvert \leq 24q+1+(2A+8|B|)q\pi\theta N^4+\frac{3A+2|B|}{3}\pi N^\delta.
\]    
\end{lem}
\begin{proof}
By Lemma \ref{Approx 2}, it suffices to show that
\[
\bigg \lvert \int_1^Ne(\theta f_t) \dd t-\int_1^N e\bigg( \frac{At^4 \theta}{24} \bigg) \dd t \bigg \rvert \leq \frac{3A+2|B|}{3}\pi N^\delta.
\]
We write
\begin{align}
\bigg \lvert &\int_1^Ne(\theta f_t) \dd t-\int_1^N e\bigg( \frac{At^4 \theta}{24} \bigg) \dd t \bigg \rvert\leq\int_1^N \left\lvert e\left(\theta\left(f_t-\frac{At^4 \theta}{24}\right)\right)-1\right\rvert dt\notag\\ 
&\quad= \int_1^N \bigg \lvert e\bigg(\theta\left(-\frac{At^3}{4}+\frac{11At^2}{24}-\frac{At}{4}+\frac{t^3B}{6}-\frac{t^2B}{2}+\frac{tB}{3}+\frac{t^2C}{2}-\frac{tC}{2}+t\right)\bigg)-1\bigg \rvert \dd t.\label{Approx 3 Step 1}
\end{align}
Since $|\theta|\leq N^{\delta-4}$ and $1 \leq t \leq N,$ we have
\[
\bigg \lvert 2\pi\theta \left(-\frac{At^3}{4}+\frac{11At^2}{24}-\frac{At}{4}+\frac{t^3B}{6}-\frac{t^2B}{2}+\frac{tB}{3}+\frac{t^2C}{2}-\frac{tC}{2}+t\right) \bigg \rvert \leq \frac{3A+2|B|}{6}\pi N^{\delta-1} \leq 1.
\]
Therefore, by the Taylor expansion of $e(x)=\exp(2\pi i x)$, we obtain
\begin{align}
\left\lvert e\left(\theta\left(f_t-\frac{At^4 \theta}{24}\right)\right)-1\right\rvert  &\leq \sum_{n=1}^\infty \frac{\left|2\pi\theta\left(f_t-\frac{At^4 \theta}{24}\right)\right|}{n!} \leq \bigg \lvert 4\pi\theta\bigg(f_t-\frac{At^4 \theta}{24}\bigg) \bigg \rvert \leq \frac{3A+2|B|}{3}\pi N^{\delta-1} . \label{Approx 3 Step 2}
\end{align}
Substituting \eqref{Approx 3 Step 2} into \eqref{Approx 3 Step 1} and integrating over $t$, we obtain the desired result.
\end{proof}

Now we are ready to prove the main lemma in this section, which is Lemma \ref{Approx 4}.
\begin{proof}[Proof of Lemma \ref{Approx 4}]
By the Binomial theorem and Lemma \ref{Approx 3}, we see that 
\begin{align*}
\bigg \lvert S_f(\alpha)^s-\bigg (\frac{V(q,a)}{24q}\int_1^N e\bigg( \frac{At^4 \theta}{24} \bigg) \dd t \bigg)^s \bigg \rvert &\leq s N^{s-1}\bigg(24q+1+(2A+8|B|)q\pi\theta N^4\\
&\quad+\frac{3A+2|B|}{3}\pi N^\delta\bigg).
\end{align*}
Since the set of $\mathfrak{M}(q,a)$ are disjoint, it follows that 
\begin{align}
&\bigg \lvert \int_{\mathfrak{M}}S_f(\alpha)^s e(-\alpha m) \dd\alpha -\int_{\mathfrak{M}} \bigg (\frac{V(q,a)}{24q}\int_1^N e\bigg( \frac{At^4 \theta}{24} \bigg) \dd t \bigg)^s  e(-\alpha m) \dd\alpha \bigg \rvert \notag \\
&\quad \leq \sum_{q\leq N^\delta}\sum\limits_{\substack{a=1 \\ (a,q)=1}}^q\int_{-N^{\delta-4}}^{N^{\delta-4}}s\left((24q+1)N^{s-1}+(2A+8|B|)q\pi\theta N^{s+3}+\frac{3A+2|B|}{3}\pi N^{\delta+s-1}\right) \dd\theta \notag\\
&\quad\leq 2s\sum_{q\leq N^\delta}\sum\limits_{\substack{a=1 \\ (a,q)=1}}^q\left(\left(1+24N^\delta\right)N^{s+\delta-5}+\left(A+4|B|\right)\pi N^{s+3\delta-5}+(A+|B|)\pi N^{s+2\delta-5}\right)\notag \\
&\quad\leq 70(A+4|B|)\pi s N^{5\delta+s-5},\notag
\end{align}
which completes the proof.
\end{proof} 

With Lemma \ref{Approx 4} established, we now turn our attention to estimating the integral in the lemma. Suppose $s \geq  17$. Let $\alpha \in \mathfrak{M}(q,a)$, $\theta = \alpha - a/q$. Define
\begin{align*}
    \mathcal{R}_{f, s}^*(m)&:=\int_{\mathfrak{M}} \bigg (\frac{V(q,a)}{24q}\int_1^N e\bigg( \frac{At^4 \theta}{24} \bigg) \dd t \bigg)^s  e(-\alpha m) \dd\alpha.
\end{align*}
We will approximate our major-arc integral by $\mathcal{R}_{f, s}^*(m)$. To achieve this, we define the following.
\begin{align}
\mathfrak{S}_{f,s}(m, Q) &: =\sum_{q\leq Q}\sum\limits_{\substack{a=1 \\ (a,q)=1}}^q\bigg(\frac{V(q,a)}{24q}\bigg)^s e\bigg(-\frac{am}{q}\bigg), \label{S (m,Q) definition} \\
v(\theta) &: = \int_1^N e\bigg (\frac{At^4 \theta}{24} \bigg) \dd t, \label{v theta definition} \\
\textrm{and} \quad J^*(m) &:= \int_{-N^{\delta-4}}^{N^{\delta-4}} \left(\left(\frac{24}{A}\right)^{1/4}v(\theta)\right)^se(-\theta m) \dd\theta. \label{J* definition}
\end{align}
Using these definitions, we then have
\begin{align}
\mathcal{R}_{f, s}^*(m)
&=\sum_{q\leq N^\delta}\sum\limits_{\substack{a=1 \\ (a,q)=1}}^q\bigg(\frac{V(q,a)}{24q}\bigg)^s e\bigg(-\frac{am}{q}\bigg)\int_{-N^{\delta-4}}^{N^{\delta-4}} v(\theta)^se(-\theta m) \dd\theta \notag\\
&=\left(\frac{24}{A}\right)^{s/4}\mathfrak{S}_{f,s}(m,N^\delta)J^*(m) \label{Approximating Major Arc Integral}.
\end{align}
Thus, to approximate $\mathcal{R}_{f, s}^*(m)$, it suffices to estimate $\mathfrak{S}_{f,s}(m,N^\delta)$ and $J^*(m)$ separately. This will be carried out in the following two sections.  

\section{Major Arcs : The Singular Series}\label{sec: Major Arcs II} 

In this section, our goal is to show that the singular series $\mathfrak{S}_{f,s}(m,N^\delta)>0$. Our strategy is as follows. We first extend the sum $\mathfrak{S}_{f,s}(m,N^\delta)$ to infinity. By demonstrating that each summand $V(q)$ for $q\in\N$ is multiplicative, we decompose the completed series into an Euler product. Next, we relate the Euler product to the counting of solutions to an equation over finite fields. Finally, using the generalized Hensel's Lemma, we show that for any finite field, there exists a solution that can be lifted to any larger field of the same characteristic. This leads us to our desired result. 

\subsection{Singular Series Completion} In this subsection, we complete the series $\mathfrak{S}_{f,s}(m,N^\delta)$. Define
\begin{align}
    \mathfrak{S}_{f,s}(m) &:= \sum_{q=1}^\infty V(q)\label{definition of S(m)}, \quad \textrm{where} \\
    \label{V Definition}
    V(q) &:= \sum\limits_{\substack{a=1 \\ (a,q)=1}}^q\bigg(\frac{V(q,a)}{24q}\bigg)^se\bigg(-\frac{am}{q}\bigg),
\end{align}
and $V(q,a)$ is given by \eqref{Definition V(q,a)}. We begin by showing that $V(q)$ is multiplicative.
\begin{lem}\label{V Multiplicativity 1}
Suppose $(a,q)=(b,r)=(q,r)=1$. Then $V(qr, ar+bq)=\frac{1}{24}V(q,a)V(r,b)$.
\end{lem}
\begin{proof}
We will focus on the case where $(24,q)=1$. Proofs for other cases follow similar arguments and thus are omitted. By Lemma \ref{Congruence}, we have
\begin{align*}
V(qr,ar+bq)=24\sum_{n=1}^{qr} e\bigg(\frac{ar+bq}{qr}f_n\bigg).
\end{align*}
Using Euclid's algorithm, we know that every residue class $m$ modulo $qr$ can be uniquely expressed as $tr+uq$ with $1\leq t\leq q$ and $1\leq u\leq r$. Therefore, we can rewrite the sum as 
\begin{align*}
     V(qr,ar+bq) &=24\sum_{t=1}^{q}\sum_{u=1}^r e\bigg(\frac{ar+bq}{qr}f_{tr+uq}\bigg)\\
     &=24\sum_{t=1}^{q}\sum_{u=1}^re\bigg(\frac{ar+bq}{qr}\cdot(f_{tr}+f_{uq})\bigg)\\
     &=24\sum_{t=1}^{q}e\bigg(\frac{a}{q}f_{tr}\bigg)\sum_{u=1}^re\bigg(\frac{b}{r}f_{uq}\bigg).
\end{align*}
Since $tr$ and $uq$ range over complete residue classes modulo $q$ and $r$ respectively, we conclude that
\begin{align*}
V(qr,ar+bq)= 24\sum_{t=1}^qe\bigg(\frac{a}{q}f_t\bigg)\sum_{u=1}^re\bigg(\frac{b}{r}f_u\bigg)=\frac{1}{24}V(q,a)V(r,b),
\end{align*}
where the last equality follows from Lemma \ref{Congruence}.
\end{proof}
\begin{lem}\label{V Multiplicativity 2}
The function $V(q)$, as defined in \eqref{V Definition}, is multiplicative.
\end{lem}
\begin{proof}
Note that $V(1)=1$. Assume $(q,r)=1$. Then, by Lemma \ref{V Multiplicativity 1},
    \begin{align*}
         V(qr) &= \sum\limits_{\substack{a=1 \\ (a,qr)=1}}^{qr}\bigg(\frac{V(qr,a)}{24qr}\bigg)^se\bigg(-\frac{am}{qr}\bigg)\\
         &=\sum\limits_{\substack{a=1 \\ (a,q)=1}}^{q}\sum\limits_{\substack{b=1 \\ (b,r)=1}}^{r}\bigg(\frac{V(qr,ar+bq)}{24qr}\bigg)^se\bigg(-\frac{ar+bq}{qr}m\bigg)\\
         &=\sum\limits_{\substack{a=1 \\ (a,q)=1}}^{q}\sum\limits_{\substack{b=1 \\ (b,r)=1}}^{r}\bigg(\frac{V(q,a)V(r,b)}{576qr}\bigg)^se\bigg(-\frac{am}{q}\bigg)e\bigg(-\frac{bm}{r}\bigg)\\
          &=V(q)V(r),
    \end{align*}
    which completes the proof.
\end{proof}
\begin{lem}\label{lem::Singular Series Extension}
Let $s \geq  17$ and $N^{\delta} \geq  e^{e^{467}} $. Then, $|\mathfrak{S}_{f,s}(m)|\leq e^{e^{468}}$. Moreover, we have
\begin{align*}
|\mathfrak{S}_{f,s}(m)-\mathfrak{S}_{f,s}(m, N^\delta)| \leq \frac{ (52A^{1/4})^s}{\left(\frac{9s}{73}-2\right) N^{\left(\frac{9s}{73}-2\right)\delta}}.
\end{align*}
\end{lem}
\begin{proof}
We begin by evaluating $V(q,a)$ using Lemma \ref{lem::upper-F(a)}. To address both cases where $(Aa,24q)=1$ and $(Aa,24q)\neq 1$, we choose $\eta = A^2$ in Lemma \ref{lem::upper-F(a)}. This gives us the bound
\begin{align*}
|V(q,a)|&\leq 2(24q)^{7/8} + 5 A^{1/4}(24q)^{1 + \frac{3.1983}{4\log\log (24q)}} \left( q^{-1} +  (24q)^{-1} +  q (24q)^{-4}\right)^{1/8} (\log q)^{1/8}\\
 &\leq 33q^{7/8}+1200A^{1/4}q^{7/8 + \frac{3.1983}{4\log\log (24q)}+\frac{\log\log q}{8\log q}}.
\end{align*}
When $q\geq  e^{e^{467}}$, we have
\[
\frac{3.1983}{4\log\log (24q)}+\frac{\log\log q}{8\log q}\leq \frac{1}{584}, 
\]
which implies 
\[
|V(q,a)|\leq 33q^{7/8}+1200A^{1/4}q^{64/73}\leq 1233A^{1/4}q^{64/73}.
\]
Therefore, we conclude that $|V(q)|\leq (52A^{1/4})^sq^{1-\frac{9s}{73}}$, where $s\geq  17$. As a result, $\mathfrak{S}_{f,s}(m)$ converges absolutely and uniformly with respect to $m$. We now have the following bound: 
\begin{align*}
|\mathfrak{S}_{f,s}(m)| \leq\sum_{q=1}^\infty|V(q)| &\leq\sum_{q=1}^{\lfloor e^{e^{467}} \rfloor}|V(q)|+ (52A^{1/4})^s\int_{e^{e^{467}}}^\infty x^{1-\frac{9s}{73}}\dd x\\
    & \leq \sum_{q=1}^{\lfloor e^{e^{467}} \rfloor} q+\frac{(52A^{1/4})^s}{\frac{9s}{73}-2}\left(e^{e^{467}}\right)^{2-\frac{9s}{73}} \leq e^{e^{468}}.
\end{align*}
Moreover, when $N^{\delta} \geq  e^{e^{467}}$,
we deduce that
\begin{align*}
    |\mathfrak{S}_{f,s}(m)-\mathfrak{S}_{f,s}(m, N^\delta)| \leq (52A^{1/4})^s\int_{N^{\delta}}^\infty x^{1-\frac{9s}{73}} \dd x \leq  \frac{ (52A^{1/4})^s}{\left(\frac{9s}{73}-2\right) N^{\left(\frac{9s}{73}-2\right)\delta}}.
\end{align*}
This completes the proof.
\end{proof}

\subsection{Counting Solution over Finite Fields} 
 Since $V(q)$ is multiplicative by Lemma \ref{V Multiplicativity 2} and $\fS(m)$ converges absolutely by Lemma \ref{lem::Singular Series Extension}, we obtain the Euler product representation
\begin{align*}
    \mathfrak{S}_{f,s}(m) =\prod_{p\text{ prime}}\sum_{k=0}^\infty V(p^k)=\prod_{p\text{ prime}} (1+V(p)+V(p^2)+\cdots).
\end{align*}
Let $1\leq n_i\leq t$. Define $\mathcal{M}_m(t,q)$ as the number of solutions to the congruence equation
\begin{align}\label{eq:sum-f(n)=m mod q}
f(n_1)+f(n_2)+\cdots+f(n_s)\equiv m \bmod q,
\end{align}
where $f(x)=A \binom{t}{4} + B\binom{t}{3} + C\binom{t}{2} + t$. To simplify the notation, we write $\mathcal{M}_m(q)=\mathcal{M}_m(q,q)$.\\

\begin{lem}\label{lem::sum-V(d)}
  For $q \in \N$, we have
  \begin{align*} 
      \sum_{d \,|\, q} V(d) =  q^{1-s} 24^{-s} \cM_m(24q,q).
  \end{align*}
\end{lem}

\begin{proof}
The proof is analogous to Lemma 6.4 in \cite{dong2024pollock} by replacing $2q$ with $24q$. Hence, we omit the details here. 
\end{proof}
It follows from Lemma \ref{lem::sum-V(d)} by choosing $q = p^k$ that 
\begin{align*}
    \fS(m) &= \prod_{p \text{ prime}} \sum_{k=0}^{\infty} V(p^k) = \prod_{ p \text{ prime}} \lim_{k \to \infty} 24^{-s}p^{k(1-s)} \cM_m(24p^k,p^k) \\ 
    &= \prod_{ p \text{ prime}} \lim_{k \to \infty} p^{k(1-s)} \cM_m(p^k),
\end{align*}
where the third equality follows from $\cM_m(24q, q) = 24^s \cM_m(q)$. Define
\begin{equation*}
    T_m(p) := \lim_{k \to \infty} p^{k(1-s)}\cM_m(p^k).
\end{equation*}
Then $ \fS(m) = \prod_{p} T_m(p)$. The following lemma gives us a bound for $T_m(p)$, which leads to an estimation for $\fS(m)$. 

\begin{lem}\label{lem::bound-Tm(p)}
    For any $s \geq 17$ and any prime $p$, we have 
    \begin{equation*}
        \left|T_m(p) - 1\right| \leq e^{se^{932}} \left(1 - \frac{1}{p}\right)\frac{p^{1 - \frac{9s}{73}}}{1 - p^{1 - \frac{9s}{73}}}. 
    \end{equation*}
\end{lem}
\begin{proof}
We consider the following cases.\medskip

\textbf{Case 1.} $p > 3$. In this case, $(24,p) =1$, so any solution $(n_1, \dots, n_s)$ of the equation
\begin{equation}
    \tilde{f}(n_1) + \dots + \tilde{f}(n_s) \equiv 24m \mod p^k
\end{equation}
where $\tilde{f}(t) := 24f(t)$, is also a solution of \eqref{eq:sum-f(n)=m mod q} and vice versa. Therefore
\begin{align}
\mathcal{M}_m(p^k)&=\frac{1}{p^k}\sum_{t=1}^{p^k}\sum_{n_1=1}^{p^k}\sum_{n_2=1}^{p^k}\cdots\sum_{n_s=1}^{p^k}e(t(\tilde{f}(n_1)+\tilde{f}(n_2)+\cdots+\tilde{f}(n_s)-24m)/p^k)\notag \\
    &=p^{(s-1)k}+\frac{1}{p^k}\sum_{t=1}^{p^k-1}e\bigg(-\frac{24mt}{p^k}\bigg)\bigg(\sum_{x=1}^{p^k}e\bigg(\frac{t\tilde{f}(x)}{p^k}\bigg)\bigg)^s. \label{Primes > 7 Step 1}
\end{align}
Each integer $1\leq t\leq p^k-1$ can be uniquely expressed as $t = bp^{k-r}$, where $(b,p)=1$, $1\leq r\leq k$, and $1\leq b\leq p^r$. Therefore, the second term in the right-hand side of equation \eqref{Primes > 7 Step 1} can be rewritten as
\begin{align}
\frac{1}{p^k}\sum_{r=1}^k&\sum\limits_{\substack{b=1 \\ (b,p)=1}}^{p^r}e\bigg(\frac{-24mb}{p^r}\bigg) \bigg(\sum_{x=1}^{p^k}e\bigg(\frac{b\tilde{f}(x)}{p^r}\bigg)\bigg)^s \notag \\
&=\frac{1}{p^k}\sum_{r=1}^k\sum\limits_{\substack{b=1 \\ (b,p)=1}}^{p^r}e\bigg(\frac{-24mb}{p^r}\bigg)\bigg(p^{k-r}\sum_{x=1}^{p^r}e\bigg(\frac{b\tilde{f}(x)}{p^r}\bigg)\bigg)^s \notag \\
&=p^{(s-1)k}\sum_{r=1}^k p^{-rs}\sum\limits_{\substack{b=1 \\ (b,p)=1}}^{p^r}e\bigg(\frac{-24mb}{p^r}\bigg) \bigg(\sum_{x=1}^{p^r}e\bigg(\frac{b\tilde{f}(x)}{p^r}\bigg)\bigg)^s.\label{Primes > 7 Step 2}
\end{align}
By combining \eqref{Primes > 7 Step 1} and \eqref{Primes > 7 Step 2}, we obtain
\begin{align*}
\left|\mathcal{M}_m(p^k)-p^{(s-1)k}\right| \leq p^{(s-1)k}\sum_{r=1}^k p^{-rs}\sum\limits_{\substack{b=1 \\ (b,p)=1}}^{p^r} \left|\sum_{x=1}^{p^r}e\bigg(\frac{b\tilde{f}(x)}{p^r}\bigg)\right|^s.
\end{align*}
Applying Lemma \ref{lem::upper-F(a)} with $X= q = p^r$ and $\eta = 1$, we have 
\begin{align*}
 \left|\sum_{x=1}^{p^r}e\bigg(\frac{b\tilde{f}(x)}{p^r}\bigg)\right| &\leq 2 p^{7r/8} + 5 p^{r + \frac{3.1983r}{4\log(3r \log p)}}\left(2p^{-r} + p^{-3r}\right)^{1/8} (r \log p)^{1/8} \\
 &\leq 2 p^{7r/8} + 5 (3 p^{-r})^{1/8} p^{r + \frac{3.1983r}{4\log(r\log p)}} p^{\frac{\log(3r\log p)}{8\log p}} \\ 
 &\leq 2 p^{7r/8} + 6 p^{7r/8 + \frac{3.1983r}{4\log(3r\log p)} + \frac{\log(r \log p)}{8\log p}}.
\end{align*}
When $r \geq 1$ and $p \geq e^{e^{467}}$, we have 
\begin{equation}
\frac{3.1983r}{4\log(3r\log p)}+\frac{\log(r\log p)}{8\log p}\leq \frac{r}{584},     
\end{equation}

It is also straight forward to check that the above inequality also holds for all $p > 3$ and $r \geq e^{467}$. Furthermore, for $3 < p \leq e^{e^{467}}$ and $1 \leq r \leq e^{467}$, trivially we have
\begin{align*}
\left|\sum_{x= 1}^{p^r} e\left(\frac{b\tilde{f}(x)}{p^r}\right)\right| \leq p^r \leq e^{e^{932}} p^{\frac{64r}{73}}.
 \end{align*}
Combining two cases, we see that for any prime $p > 3$ and any $r\in \N$, we have 
\begin{equation}
    \left|\sum_{x= 1}^{p^r} e\left(\frac{b\tilde{f}(x)}{p^r}\right)\right| \leq e^{e^{932}}p^{\frac{64r}{73}}. 
\end{equation}

\textbf{Case 2.} $p = 2$. In this case, the equation \eqref{eq:sum-f(n)=m mod q} is equivalent to 
\[ 3 f(n_1) + \dots + 3 f(n_s) \equiv 3m \mod 2^k.\]
Therefore, we can write 
\begin{align}
\mathcal{M}_m(2^k)&=\frac{1}{2^k}\sum_{t=1}^{p^k}\sum_{n_1=1}^{2^k}\cdots\sum_{n_s=1}^{2^k}e(t(3\tilde{f}(n_1)+\cdots+3\tilde{f}(n_s)-3m)/2^k)\notag \\
&=2^{(s-1)k}+\frac{1}{2^k}\sum_{t=1}^{2^k-1}e\bigg(-\frac{3mt}{2^k}\bigg)\bigg(\sum_{x=1}^{2^k}e\bigg(\frac{t\tilde{f}(x)}{2^{k+3}}\bigg)\bigg)^s, \label{Prime=2 Step 1}
\end{align}
Similar with Case 1, it follows that 
\begin{equation*}
    \left|\cM_m(2^k) - 2^{(s-1)k}\right| \leq 2^{(s-1)k} \sum_{r = 1}^k 2^{-rs} \sum_{\substack{b = 1 \\ (b,2) = 1}}^{2^r} \left|\sum_{x = 1}^{2^r} e\left(\frac{t \tilde{f}(x)}{2^{r+3}}\right)\right|^s  
\end{equation*}
for any $r \in \N$. Applying Lemma \ref{lem::upper-F(a)} with $\eta = 1, X = 2^r$, and $q = 2^{r+3}$, we have
\begin{align*}
    \left|\sum_{x = 1}^{2^r} e\left(\frac{b \tilde{f}(x)}{2^{r+3}}\right)\right| &\leq 2 \cdot 2^{7r/8} + 5 \cdot 2^{r + \frac{3.1983r}{4\log(3r\log 2)}}(2^{-r-2} + 2^{-r} + 2^{2-3r})^{1/8} \left(\log(2^{r+3})\right)^{1/8}  \\
    &= 2 \cdot 2^{7r/8} + 6 \cdot 2^{7r/8 + \frac{3.1983r}{\log(3r \log 2)} + \frac{\log\left((r+3)\log 2\right)}{8 \log 2}} \leq 2^{e^{467}} \cdot 2^{\frac{64r}{73}}.
\end{align*}
\textbf{Case 3.} $ p =3$. This case is similar with Case 2, in which we use the equivalency between \eqref{eq:sum-f(n)=m mod q} and 
\[ 8 f(n_1) + \dots + 8 f(n_s) \equiv 8m \mod 3^k.\]
In particular, we obtain 
\begin{align*}
    \left|\cM_m(3^k) - 3^{(s-1)k}\right| \leq 3^{(s-1)k} \sum_{r = 1}^k 3^{-rs} \sum_{\substack{b = 1 \\ (b,3) = 1}}^{3^r} \left|\sum_{x = 1}^{3^r} e \left(\frac{b\tilde{f}(x)}{3^{r+1}}\right)\right| 
\end{align*}
Again, by applying Lemma \ref{eq:bound.F(alpha)} with $\eta = 1, X = 3^r,$ and $q =3^{r+1}$, we have for any $r \in \N$, 
\begin{align*}
 \left|\sum_{x = 1}^{3^r} e \left(\frac{b\tilde{f}(x)}{3^{r+1}}\right)\right|  \leq e^{e^{467}} 3^{\frac{64r}{73}}.   
\end{align*}
Combing all three cases, we obtain that for $s \geq 17$ and for any prime $p$, 
\begin{equation*}
    |T_m(p) - 1| \leq \sum_{r = 1}^{\infty} p^{-rs} \sum_{\substack{b = 1 \\ (b,p) = 1}}^{p^r} e^{se^{932}} p^{\frac{64sr}{73}}  = e^{se^{932}} \left(1 - \frac{1}{p}\right)\frac{p^{1 - \frac{9s}{73}}}{1 - p^{1 - \frac{9s}{73}}}. 
\end{equation*}
\end{proof}

\subsection{Hensel Lifting} 

In this subsection, we aim to provide a lower bound for $\cM_m(p^k)$ for any prime $p$ and $k\in\N$. First, we consider the case when $k = 1$. In this case, $\cM_m(p)$ is the number of solutions to the equation 
\begin{equation}\label{eq:sum-f(n)=m-mod p}
    f(n_1) + f(n_2) + \dots + f(n_s) \equiv m \mod p, 
\end{equation}
where $p \geq 7$ and $1 \leq n_i \leq p$ for $i = 1, 2, \dots, s$. Let $\cM_m^*(p)$ be the number of solutions to \eqref{eq:sum-f(n)=m-mod p} such that $(f(n_1),p) = (f'(n_1), p) = 1$. The following lemma states that there is indeed at least one such solution. 

\begin{lem}\label{lem:M-star-p}
For $p \geq 11$ and $s \geq 17$, we have $\cM_m^*(p) \geq 1$.     
\end{lem}

To prove Lemma \ref{lem:M-star-p}, we use the following well-known Cauchy-Davenport inequality in additive combinatorics, see \cite{davenport1935addition}.

\begin{thm}[Cauchy-Davenport inequality]\label{thm: Cauchy Davenport inequality}
 Suppose that $\cA$ is a set of $r$ residue classes modulo $q$, and that $\cB$ is a set
of $s$ such classes. Suppose further that $0 \in \cB$, and that whenever $b \in \cB$ and $q \nmid b$, one has $(b, q) = 1$. Then   
\[ |\cA + \cB| \geq \min\{q, |\cA| + |\cB| - 1\}.\]
\end{thm}

\begin{proof}[Proof of Lemma \ref{lem:M-star-p}]
 Let 
\[ \cA = \{ f(n) ~\mathrm{ mod }~ p : 1 \leq n \leq p, (f(n), p) = (f'(n), p) = 1 \}\]
and $\cB = \cA \cup \{0\}$. Since there are at most $7$ possibilities of $n$ such that $(f(n),p) = p$ or $(f'(n),p) = p$, we have $|\cA| \geq \left\lceil\frac{p-7}{4}\right\rceil$ and $|\cB| = |\cA| + 1$.   
 By applying Cauchy-Davenport inequality and induction on $s$, we obtain 
\[ |\cA +  (s-1)\cB| \geq \min\left\{p,  s\left\lceil \frac{p-7}{4}\right\rceil\right\} = p\]
when $p \geq 11$ and $s \geq 17$. Hence, there exist $a \in A$ and $b_2, \dots, b_s \in \cB$ such that 
\[f(a) + f(b_2) + \dots + f(b_s) \equiv m~(\mathrm{mod}~p).\]
By definition of $\cA$, we have $(f(a),p) = (f'(a),p) = 1$. Thus $\cM_m^*(p) \geq 1$.
\end{proof}

Using Lemma \ref{lem:M-star-p}  together with Hensel's lemma, we immediately obtain the following corollary.

\begin{coro}\label{coro:lower-bound-M_m(p^k)}
For any prime $p \geq 11$, any integers $k \geq 2$ and $s \geq 17$, we have 
$$\cM_m(p^k) \geq p^{(k-1)(s-1)}.$$     
\end{coro}

Now we are left to discuss  lower bounds for $T_m(p)$ for $p\leq 7$. To this end, we rely on the generalized Hensel's Lemma.
\begin{lem}[generalized Hensel's Lemma, \cite{niven1991}]
    Let $f(x)$ be a polynomial with integral coefficients. Suppose that there exist $j,\tau\in\N$ such that $f'(a)\equiv 0\bmod p^j$, $p^\tau||f'(a)$ and that $j\geq 2\tau+1$. If $b\equiv a\bmod p^{j-\tau}$, then $f(b)\equiv f(a)\bmod p^j$ and $p^\tau||f'(b)$. Moreover, there is a unique $a_0\bmod p$ such that $f(a+a_0p^{j-\tau})\equiv 0\bmod p^{j+1}$. 
\end{lem}

\begin{lem}\label{lem:lower-bound-T_m(p)}
   For $s \geq 17$ and any $f$ satisfying \eqref{defn: f(n)}, we have 
   \[ T_m(p) > p^{1-s},\]
for $p\geq 11$. For $p<11$, define
\begin{align}
\tau = \max_{p<11} \min\{\tau\in \N: f'(y)\neq 0\bmod p^\tau \text{ for some integer $y$}\}.\label{def of tau}
\end{align}
Then,
   \[
T_m(p) > p^{(2\tau+1)(1-s)}.
\]
In particular, when $f$ is one of the target polynomials $f_1,f_2,f_3$, defined in \eqref{target polynomials}, we have
\[ T_m(p) > p^{1-s},\]
for $p\geq 3$, and 
\[
T_m(2) > 2^{5(1-s)}.
\]
\end{lem}
\begin{rem}
    Since $f'$ is a degree-three polynomial, the integer $y$ in \eqref{def of tau} always exists for infinitely many $\tau\in\N$ by the property of the ring of $p$-adic integers.
\end{rem}
\begin{proof}

When $p\geq 11$, by Corollary \ref{coro:lower-bound-M_m(p^k)}, we have 
\[ T_m(p) = \lim_{k \to \infty} p^{k(1-s)} \cM_m(p^k) \geq p^{1-s} > 0\]
for any $p \geq 11$ and $s \geq 17$. Now assume $p<11$. We will focus on $p=2$, as the other small $p$ cases follow a similar argument. 
 
Fix $y_1\in\N$ such that $f'(y_1)\neq 0\bmod 2^\tau$ for some $\tau\geq 1$. If $\tau=1$, then by using Hensel's lemma, we obtain $T_m(2)>2^{1-s}$. Otherwise, define 
\[ \cA = \{ f(n) ~\mathrm{ mod }~ 2^{2\tau+1} : 1 \leq n \leq 2^{2\tau+1}, (f(n), 2) = 1 \}\]
and $\cB = \cA \cup \{0\}$. Then, using similar arguments as in the proof of Lemma \ref{lem:M-star-p}, and since $s\geq 17$, we obtain
\[ |\cA +  (s-2)\cB| \geq \min\left\{2^{2\tau+1},  (s-2)\left\lceil \frac{2^{2\tau+1}-4}{4}\right\rceil\right\} = 2^{2\tau+1}.\]
Therefore, there always exist $y_2,\dots,y_s$ such that 
\[
f(y_2)+\dots+f(y_s)\equiv m-f(y_1)\bmod 2^{2\tau+1}.
\]
Now apply generalized Hensel's Lemma, for any $k\geq 2\tau+1$, $(y_1,y_2,\cdots,y_s)$ can be lifted to $2^{(k-2\tau-1)(s-1)}$ solutions to 
\begin{align}\label{solutions to 2 to the kth power}
f(n_1)+f(n_2)+\cdots+f(n_s)\equiv m \bmod p^k.
\end{align}
 Thus, the first part of the lemma is proved.

To prove the remaining statement, it suffices to find $\tau$ for each $f_i$ and $p<11$. We will present $f=f_1$ here, as the other two polynomials follow similar arguments. Assume $p=2$. Recall that $f(0)=0$ and $f(1)=1$. Fix $n_1=1$. Note that $f'(1)=4$, which is congruent to $0$ modulo $4$ but not congruent to $0$ modulo $8$. Then, by the argument above, $(1,y_1,\cdots,y_s)$ can be lifted to $2^{(k-5)(s-1)}$ solutions to \eqref{solutions to 2 to the kth power}. Thus, 
\[
T_m(2)\geq 2^{5(1-s)}.
\]

Finally, for $3\leq p\leq 7$, note that $f'(1)=4$ is not congruent to $0$ modulo $p$, so 
\[
T_m(p)\geq p^{1-s}. 
\]
\end{proof}
\subsection{A lower bound for $\mathfrak{S}_{f,s}(m)$}
We are ready to establish a lower bound for $\mathfrak{S}_{f,s}(m)$.
\begin{lem}\label{lower bound for S(m)}
For any polynomial $f$ in the form of \eqref{defn: f(n)}, let $\tau$ be as defined in \eqref{def of tau}. We then have $\mathfrak{S}_{f,s}(m)>0$. More precisely, 
\begin{align}\label{Sigma Bound for general f}
 \mathfrak{S}_{f,s}(m) \geq 2^{\frac{z^2}{146-9s}}\exp((12\tau+1.03z)(1-s))>0
\end{align}
uniformly, where $z=(2 e^{se^{932}}+1)^{\frac{73}{9s - 21}}$. In particular, for the three target polynomials $f_1,f_2$, and $f_3$, we have
\begin{align}\label{Sigma Bound}
 \mathfrak{S}_{f,s}(m) \geq 2^{\frac{z^2}{146-9s}}\exp(1.03z(1-s))>0.
\end{align}
\end{lem}
\begin{proof}
Define $\theta_p := T_m(p)-1$ for prime $p$ and let $w = \prod_{p > z} T_{m}(p)^{-1}$. From Lemma \ref{lem::bound-Tm(p)}, it follows that for $p > z$, $\left|\theta_p\right|\leq\frac{1}{2}$.
By the Taylor expansion, we write 
\begin{align}\label{Taylor Expansion}
    \log w  
    &= - \sum_{p > z} \log(1 + \theta_{p})=\sum_{p > z} \left(\theta_{p} + \frac{\theta^{2}_{p}}{2} + \frac{\theta^{3}_{p}}{3} + \cdots \right).
\end{align}
Thus, by Lemma \ref{lem::bound-Tm(p)} and \eqref{Taylor Expansion}, we have
\begin{align*}
    \log w &\leq 4 \log 2 \cdot e^{se^{932}}\sum_{p > z} p^{1-\frac{9s}{73}} \leq 4 \log 2 \cdot e^{se^{932}} \int^{\infty}_{z} x^{1-\frac{9s}{73}}  \dd x
= \frac{146z^{125/73}\log 2}{9s-146}.
\end{align*}
Exponentiating both sides of the above inequality, we obtain
\[
 w \leq  2^{\frac{146z^{125/73}}{9s-146}},
\]
which implies that
\begin{align}
\prod_{p > z} T_{m}(p)\geq 2^{\frac{146z^{125/73}}{146-9s}}\geq 2^{\frac{z^2}{146-9s}} .  \label{Tm(p) p>M}  
\end{align}
For $p\leq z$, Lemma \ref{lem:lower-bound-T_m(p)} shows that
\begin{align}
\prod_{p \leq z} T_m(p) &\geq 210^{(2\tau+1)(1-s)}\prod_{11\leq p \leq z} p^{1-s}\notag\\
& \geq 210^{2\tau(1-s)}\exp\bigg((1-s)\left(0.0242334\frac{z}{\log z}+z\right)\bigg)\\
&\geq \exp((1.03z+12\tau)(1-s))\label{Tm(p) p<= M},
\end{align}
where the second inequality follows from Rosser and Schoenfeld's work \cite{Rosser1975} on upper bounds for $\theta(x)$. For $f_1, f_2,$ and $f_3$, we have
\[
\prod_{p \leq z} T_m(p)\geq \exp(1.03z(1-s)).
\]
Combining inequalities \eqref{Tm(p) p>M} and \eqref{Tm(p) p<= M}, we have for a general $f$ satisfying \eqref{defn: f(n)},
\begin{align*}
     \mathfrak{S}_{f,s}(m)&\geq 2^{\frac{z^2}{146-9s}}\exp((1.03z+12\tau)(1-s))>0.
\end{align*}
For the target polynomials $f_1,f_2,f_3$, we have
\[
\mathfrak{S}_{f_i,s}(m)\geq 2^{\frac{z^2}{146-9s}}\exp(1.03z(1-s))>0,
\]
which completes the proof.
\end{proof}

\section{Major Arcs : The Singular Integral}\label{sec: Major Arcs Sec III}  In this section, we estimate the singular integral $J^{*}(m)$ given by \eqref{J* definition}. In light of the discussion in Lemma \ref{lem::Singular Series Extension}, we assume $N^{\delta} \geq  e^{e^{467}}$. We proceed as follows: First, we estimate the integral $v(\theta)$ in the integrand of $J^{*}(m)$ by the finite sum $v_1(\theta)$, defined in \eqref{v_1 definition} below. By replacing $v(\theta)$ in $J^{*}(m)$ with $v_1(\theta)$, we may extend the range of this new integral to a unit interval, and then approximate it using properties of Gamma functions. 
\subsection{Approximation of $v(\theta)$} Let $N_0 =\frac{A}{24}N^4$ where $N$ is given by \eqref{Defining N}. Define
\begin{align}
    v_1(\theta):= \frac{1}{4}\sum_{1 \leq n \leq N_0} n^{-3/4}e(\theta n) \quad \textrm{and} \quad v_2(\theta) := \int_0^{N_0^{1/4}} e(\theta t^4) \dd t.\label{v_1 definition}
\end{align}
\begin{lem}\label{v_3(theta)-v_1(theta) bound lemma}
Let $v(\theta)$ and $v_1(\theta)$ be as defined in \eqref{v theta definition} and \eqref{v_1 definition} respectively. Then
\[\bigg \lvert v_1(\theta)-\left(\frac{A}{24}\right)^{1/4}v(\theta)\bigg \rvert \leq AN^\delta.
\]
\end{lem}
\begin{proof}
We have
\begin{align}
   \bigg \lvert v_2(\theta) -\left(\frac{A}{24}\right)^{1/4}v(\theta) \bigg \rvert &\leq \int_0^{\sqrt[4]{\frac{A}{24}}} \left|e(\theta t^4)\right| \dd t=\left ( \frac{A}{24}\right)^{1/4}\label{|v_2(theta)-v_1(theta)|}.
\end{align}
Since $f(n) = n^{-3/4}$ is a decreasing function, we have
\begin{align*}
\int_1^{N_0+1}t^{-3/4}\dd t \leq \sum_{1 \leq n \leq N_0}n^{-3/4}\leq \int_0^{N_0}t^{-3/4} \dd t.
\end{align*}
Therefore, we get
\begin{align}\label{sum-integral bound}
\left|\frac{1}{4}\sum_{1 \leq n \leq N_0}n^{-3/4}-\frac{1}{4}\int_1^{N_0}t^{-3/4} \dd t\right| \leq \frac{1}{2} \left|\int_0^1 t^{-3/4} \dd t\right|= 2.
\end{align}
Since $\frac{1}{4}\int_1^{N_0}t^{-3/4}dt=N_0^{1/4}-1,$ we obtain
\begin{align}
\bigg|\frac{1}{4}\sum_{1 \leq n \leq N_0} n^{-3/4}-N_0^{1/4}\bigg|\leq 3.\label{intermediate step}
\end{align}
Using partial summation, we have
\begin{align}\label{Partial Summation v_1}
    v_1(\theta)=e(N_0\theta)\bigg(\frac{1}{4}\sum_{1 \leq n \leq N_0} n^{-3/4}\bigg)-2\pi i\theta\int_1^{N_0} \bigg(\frac{1}{4}\sum_{1 \leq n \leq t} n^{-3/4}\bigg)e(\theta t) \dd t.
\end{align}
On the other hand, using integration by parts and a change of variables, we obtain
\begin{align}\label{By Parts v_2}
v_2(\theta) = e(N_0\theta)N_0^{1/4}-2\pi i\theta\int_0^{N_0} t^{1/4} e(\theta t) \dd t.
\end{align} 
Combining \eqref{sum-integral bound}, \eqref{intermediate step}, \eqref{Partial Summation v_1} and \eqref{By Parts v_2}, we deduce that
\begin{align}
\left|v_1(\theta)-v_2(\theta) \right| &\leq \bigg \lvert \frac{1}{4}\sum_{1 \leq n \leq N_0} n^{-3/4}-N_0^{1/4}\bigg \rvert +2\pi|\theta|\bigg (1+ \int_1^{N_0} \bigg \lvert t^{1/4} - \bigg(\frac{1}{4}\sum_{1 \leq n \leq t} n^{-3/4}\bigg)\bigg \rvert \dd t  \bigg) \notag\\
&\leq3+2\pi|\theta|(3N_0-2) \label{v_1-v_2}.
\end{align}
Recall that $|\theta|<N^{\delta-4}$ and $N^{\delta}\geq  e^{e^{467}}$. Putting together \eqref{|v_2(theta)-v_1(theta)|} and \eqref{v_1-v_2}, we have
\begin{align*}
\bigg \lvert v_1(\theta)-\left(\frac{A}{24}\right)^{1/4}v(\theta)\bigg \rvert \leq  3+\pi(6N_0-4)N^{\delta-4}+\sqrt[4]{A/24} \leq AN^\delta,
\end{align*}
which completes the proof.
\end{proof}

\begin{lem}\label{V_1 theta Estimation}
Let $|\theta| \leq \frac{1}{2}$ and $v_1(\theta)$ be as in \eqref{v_1 definition}. Then $|v_1(\theta)|\leq \min \{2m^{1/4}, 2 |\theta|^{-1/4} \}$.
\end{lem}
\begin{proof}
If $|\theta|\leq m^{-1}$, the trivial bound suffices. Assume $|\theta| > m^{-1}$ and let $M=[ |\theta|^{-1}]$. Then, the contribution to $v_1(\theta)$ from the terms $n\leq M$ is bounded by $ M^{1/4}\leq |\theta|^{-1/4}$. 

When $n> M$, define $S_n = \sum_{1 \leq r \leq n}e(\theta r)$ and $c_n = \frac{1}{4}n^{-3/4}$. We can express the sum as 
\begin{align*}
 \frac{1}{4}\sum_{n=M+1}^{N_0} n^{-3/4}e(\theta n)=c_{N_0+1}S_{N_0}-c_{M+1}S_M+\sum_{n=M+1}^{N_0}(c_n-c_{n+1})S_n.
\end{align*}
By Lemma \ref{lem::Geometric Sum}, we know that $\lvert S_n\rvert\leq\frac{1}{2\lvert\theta\rvert}$. Since $c_n$ is strictly decreasing, we have
\begin{align*}
\frac{1}{4}\sum_{n=M+1}^{N_0} n^{-3/4}e(\theta n)&= -c_{M+1}S_M+\sum_{n=M+1}^{N_0-1}(c_n-c_{n+1})S_n+c_{N_0}S_{N_0}\\
&\leq c_{M+1}S_M+\sum_{n=M+1}^{N_0-1}(c_n-c_{n+1})\frac{1}{2\lvert\theta\rvert}+c_{N_0}\frac{1}{2\lvert\theta\rvert}\\
&\leq 2c_{M+1}\frac{1}{2\lvert\theta\rvert}\leq \lvert\theta\rvert^{-1/4}.
\end{align*} 
Adding the contributions from the two parts, we obtain the desired bound $|v_1(\theta)|\leq 2|\theta|^{-1/4}$. This completes the proof.
\end{proof}
Next, we introduce another useful lemma from \cite{dong2024pollock} with a slight modification.
\begin{lem}[Lemma 7.3, \cite{dong2024pollock}]\label{gamma approximation}
    Suppose $\alpha,\beta$ are real numbers with $\alpha \geq  \beta>0,\beta<1$. Then
    \begin{align*}
        \left|\sum_{n=1}^{m-1} n^{\beta-1}(m-n)^{\alpha-1}-m^{\beta+\alpha-1}\bigg(\frac{\Gamma(\beta)\Gamma(\alpha)}{\Gamma(\beta+\alpha)}\bigg)\right|\leq \frac{2}{\beta}m^{\alpha-1}{}_2F_1\left(\beta,1-\alpha,1+\beta,\frac{1}{m} \right),
    \end{align*}
    where $\Gamma$ is the Gamma function and ${}_2F_1$ is the hypergeometric function. When $\beta=\frac{1}{4}$,     
\[  \left|\sum_{n=1}^{m-1} n^{\beta-1}(m-n)^{\alpha-1}-m^{\beta+\alpha-1}\bigg(\frac{\Gamma(\beta)\Gamma(\alpha)}{\Gamma(\beta+\alpha)}\bigg)\right| \leq 12m^{\alpha-1}.
\]
\end{lem}
\begin{proof}
 When $\beta=\frac{1}{4}$, ${}_2F_1\left(\beta,1-\alpha,1+\beta,\frac{1}{m}\right) 
<2$, so the proof follows exactly as in \cite{dong2024pollock}.
\end{proof}
\subsection{Approximating the Singular Integral} First, let's consider the integral
\begin{align}
J_1^*(m) &:= \int_{-N^{\delta-3}}^{N^{\delta-3}} v_1(\theta)^se(-\theta m) \dd \theta. \label{J_1* definition}
\end{align}
Therefore, by Lemma \ref{v_3(theta)-v_1(theta) bound lemma},
\begin{align}
|J_1^*(m)-J^*(m)|& \leq \int_{-N^{\delta-4}}^{N^{\delta-4}}\bigg \lvert v_1(\theta)^s-\left(\left(\frac{A}{24}\right)^{1/4}v(\theta)\right)^s\bigg \rvert \dd\theta\notag\\
&\leq AN^{\delta}\int_{-N^{\delta-4}}^{N^{\delta-4}}\sum_{j=1}^s \lvert v_1(\theta) \rvert^{s-j}\cdot \bigg \lvert \left(\frac{A}{24}\right)^{1/4}v(\theta)\bigg \rvert^{j-1} \dd\theta\notag\\
&\leq 2AsN^\delta N_0^{(s-1)/4}N^{\delta-4} \leq 2As\left(\frac{A}{24}\right)^{s/4}N^{s-5+2\delta}\label{J_1*(m)-J*(m)}.
\end{align}

Now, we extend the integral $J^{*}_1(m)$ to an integral over a unit interval. We define
\begin{align}
J_1(m) := \int_{-1/2}^{1/2}v_1(\theta)^se(-\theta m)\dd\theta.\label{J_1(m)}
\end{align}
We aim to approximate $J^{*}(m)$ with $J_1(m)$. By Lemma \ref{V_1 theta Estimation} and the fact that $N^{\delta-4}< \frac{1}{2}$, we have
\begin{align}
|J_1(m)-J_1^*(m)| &\leq \bigg|\int_{-1/2}^{-N^{\delta-4}} v_1(\theta)^se(-\theta m)\dd\theta\bigg|+\bigg|\int_{N^{\delta-4}}^{1/2} v_1(\theta)^se(-\theta m)\dd\theta\bigg|\notag\\
    &\leq 2\int_{N^{\delta-4}}^{1/2}\bigg(\min \bigg \{2m^{1/4}, 2 |\theta|^{-1/4} \bigg \}\bigg)^s \dd\theta\notag\\
    &\leq 2^{s+1}\int_{N^{\delta-4}}^{1/2}\theta^{-s/4}\dd\theta \leq \frac{2^{s+3}}{s-4}N^{\frac{4\delta-s\delta+4s-16}{4}}.\label{J_1(m)-J_^*(m)}
\end{align}
Combining \eqref{J_1*(m)-J*(m)} with \eqref{J_1(m)-J_^*(m)}, it follows that
\begin{align}
|J_1(m)-J^*(m)| \leq 2As\left(\frac{A}{24}\right)^{s/4}N^{s-5+2\delta}+\frac{2^{s+3}}{s-4}N^{\frac{4\delta-s\delta+4s-16}{4}}\label{J_1(m)-J^*(m)}.
\end{align} 
\begin{lem}\label{J_1(m) approximation bound}
Let $s\geq  2$. Then 
\begin{align*}
\bigg|J_1(m,s)-\Gamma\bigg(\frac{5}{4}\bigg)^s\Gamma\bigg(\frac{s}{4}\bigg)^{-1}m^{s/4-1}\bigg|\leq m^{(s-1)/4-1}.
\end{align*}
\end{lem}
\begin{proof}
    We write
\begin{align*}
    J_1(m)=J_1(m,s) &= \int_{-1/2}^{1/2}v_1(\theta)^se(-\theta m) \dd\theta\\
&=4^{-s}\sum_{n_1=1}^{N_0}\cdots\sum_{n_s=1}^{N_0}(n_1n_2\cdots n_s)^{-3/4}\int_{-1/2}^{1/2}e(n_1+\cdots+n_s-m) \dd\theta\\
    &=4^{-s}\mathop{\sum_{n_1=1}^{N_0}\cdots\sum_{n_s=1}^{N_0}}_{n_1+\cdots+n_s=m} (n_1n_2\cdots n_s)^{-3/4}.
\end{align*}
When $s=2$, 
\begin{align*}
J_1(m,2)&=4^{-2}\mathop{\sum_{n_1=1}^{m-1}\sum_{n_2=1}^{m-1}}_{n_1+n_2=m}(n_1n_2)^{-3/4}=4^{-2}\sum_{n_1=1}^{m-1} n_1^{-3/4}(m-n_1)^{-3/4}.
\end{align*}
Applying Lemma \ref{gamma approximation} with $\alpha=\beta=\frac{1}{4}$, we deduce that
\begin{align*}
&\bigg|J_1(m,2)-4^{-2}m^{-1/4}\frac{\Gamma(\frac{1}{4})\Gamma(\frac{1}{4})}{\Gamma(\frac{1}{2})} \bigg|\leq \frac{3}{4}m^{-3/4}.
\end{align*}
Rewriting the above, we obtain
\[
\bigg|J_1(m,2)-\Gamma\bigg(\frac{5}{4}\bigg)^2\Gamma\bigg(\frac{1}{2}\bigg)^{-1}m^{-1/4}\bigg| \leq \frac{3}{4}m^{-3/4},
\]
and thus the lemma holds for $s=2$. Now, suppose the lemma holds for some $s\geq  2$. Note that 
\begin{align*}
    J_1(m,s+1)=\frac{1}{4}\sum_{n=1}^{m-1}n^{-3/4}J_1(m-n,s).
\end{align*}
By assumption,
\begin{align} \label{J1 bound Step 1}
\bigg|J_1(m-n,s)-\Gamma\bigg(\frac{5}{4}\bigg)^s\Gamma\bigg(\frac{s}{4}\bigg)^{-1}(m-n)^{s/4-1}\bigg|\leq (m-n)^{(s-1)/4-1}.
\end{align}
Moreover, by Lemma \ref{gamma approximation} with $\beta=1/4$ and $\alpha=s/4$, and the fact that $\Gamma(5/4)=4\Gamma(1/4)$,
\begin{align}\label{J1 bound Step 2}
\bigg|\frac{1}{4} \sum_{n=1}^{m-1}n^{-3/4}\Gamma\bigg(\frac{5}{4}\bigg)^s\Gamma\bigg(\frac{s}{4}\bigg)^{-1}(m-n)^{s/4-1}  &-\Gamma\bigg(\frac{5}{4}\bigg)^{s+1}\Gamma\bigg(\frac{s+1}{4}\bigg)^{-1}m^{(s+1)/4-1}\bigg| \notag \\
&\leq 3 \cdot \Gamma\bigg(\frac{5}{4}\bigg)^s\Gamma\bigg(\frac{s}{4}\bigg)^{-1}m^{s/4-1}.
\end{align}
Therefore combining \eqref{J1 bound Step 1} and \eqref{J1 bound Step 2}, and applying Lemma \ref{gamma approximation} we obtain
\begin{align*}
\bigg|J_1(m,s+1)&-\Gamma\bigg(\frac{5}{4}\bigg)^{s+1}\Gamma\bigg(\frac{s+1}{4}\bigg)^{-1}m^{(s+1)/4-1}\bigg|\\
&\leq \frac{1}{4}\sum_{n=1}^{m-1}n^{-3/4}(m-n)^{(s-1)/4-1}+3 \cdot \Gamma\bigg(\frac{5}{4}\bigg)^s\Gamma\bigg(\frac{s}{4}\bigg)^{-1}m^{s/4-1}\\
&\leq \frac{1}{4}m^{s/4-1}\Gamma\bigg(\frac{1}{4}\bigg)\Gamma\bigg(\frac{s-1}{4}\bigg)\Gamma\bigg(\frac{s}{4}\bigg)^{-1}+3m^{(s-1)/4-1}\\
&\quad +3 \cdot \Gamma\bigg(\frac{5}{4}\bigg)^s\Gamma\bigg(\frac{s}{4}\bigg)^{-1}m^{s/4-1}\leq m^{s/4-1}.
\end{align*}
By induction, the proof follows.
\end{proof}
\section{Asymptotic Results for Representations as Sums of numbers in Target Polynomial forms}\label{sec: Proof of Main Theorem}
Our goal in this section is to prove Theorem \ref{thm: representations general for deg 4}. Our first intermediate result is as follows.
\begin{thm}\label{thm: main Theorem for formula}
For $m,s \in \mathbb{N}$, let $\mathcal{R}_{f,s}(m)$ denote the number of ways of representations of $m$ as the sum of $s$ numbers represented by the polynomial $f$ defined in \eqref{defn: f(n)}. Then for any $s\geq 17$, $0<\delta<\frac{1}{5}$, and $m>\frac{A}{24}(e^{e^{468}})^{4/\delta}$,
\begin{align}
\bigg|\mathcal{R}_{f,s}(m)& -\bigg(\frac{24}{A}\bigg)^{s/4}\mathfrak{S}_{f,s}(m)\Gamma\bigg(\frac{5}{4}\bigg)^{s}\Gamma\bigg(\frac{s}{4}\bigg)^{-1}m^{s/4-1}\bigg| \notag \\
     &\leq 70(A+4|B|)\cdot \pi s \left(\frac{81m}{A}\right)^{\frac{5\delta+s-5}{4}}+\frac{230^s\cdot 730}{9s-146}A^{\frac{9\delta s-146\delta}{292}}m^{\frac{73s-292-9\delta s+146\delta}{292}}\notag\\
    &\quad+2Ase^{e^{468}}\left(\frac{81m}{A}\right)^{\frac{s-5+2\delta}{4}}+\bigg(\frac{24}{A}\bigg)^{s/4} \frac{2^{s+3}}{s-4}e^{e^{468}}{\left(\frac{81m}{A}\right)}^{\frac{4\delta-s\delta+4s-16}{16}}\notag\\
    &\quad+\left(\frac{24}{A}\right)^{s/4}\left(e^{e^{468}}\cdot m^{\frac{s-5}{4}}\right)+10^6 \cdot 11^{s-16} (\log m)^{\frac{s-16}{8}} \left(\frac{81m}{A}\right)^{\frac{s-4}{4} - \frac{\delta(s-16)}{32} + \frac{s}{4\log\log (m/A)-8}} \label{  1st Theorem Eq},
\end{align}
where $\mathfrak{S}_{f,s}(m)$ is defined in \eqref{definition of S(m)} and satisfies \eqref{Sigma Bound for general f}. 
\end{thm}
\begin{proof} For $s>3, \lvert\theta\rvert\leq\frac{1}{2}$, we have $\lvert (\frac{A}{24})^{1/4}v(\theta)\rvert\leq\min\{\left(\frac{A}{24}\right)^{1/4}N,2\lvert\theta\rvert^{-1/4}\}$ following a similar computation as in Lemma \ref{V_1 theta Estimation}. Therefore, we obtain
\begin{align}
|J^*(m)|&= \left|\int_{-N^{\delta-4}}^{N^{\delta-4}}\left(\frac{A}{24}\right)^{s/4}v^s(\theta)e(-\theta m) \dd\theta\right| \leq \left( \frac{2^{s+3}}{s-4}+  \frac{2}{5^s}\right)m^{s/4-1}\label{|J^*(m)|}.
\end{align}
Now, we show that in \eqref{Approximating Major Arc Integral}, one may replace $J^{*}(m)$ by $J_1(m)$ while allowing a small error. Indeed, combining Lemma \ref{lem::Singular Series Extension}, \eqref{J_1(m)-J^*(m)} and \eqref{|J^*(m)|}, we deduce that
\begin{align*}
\bigg|\mathcal{R}_{f, s}^*(m)&- \bigg(\frac{24}{A}\bigg)^{s/4}\mathfrak{S}_{f,s}(m)J_1(m)\bigg| \notag \\
    &\leq \bigg(\frac{24}{A}\bigg)^{s/4}\bigg(|J^*(m)|\cdot |\mathfrak{S}_{f,s}(m)-\mathfrak{S}_{f,s}(m, N^\delta)|+|\mathfrak{S}_{f,s}(m)||J_1(m)-J^*(m)|\bigg)\notag\\
    &\leq\bigg(\frac{24}{A}\bigg)^{s/4} \bigg(\left( \frac{2^{s+3}}{s-4}+  \frac{2}{5^s}\right)\frac{ (52A^{1/4})^s}{\left(\frac{9s}{73}-2\right) N^{\left(\frac{9s}{73}-2\right)\delta}}m^{\frac{s}{4}-1}+2e^{e^{468}}As\left(\frac{A}{24}\right)^{s/4}N^{s-5+2\delta}\notag\\
    &\quad+\frac{2^{s+3}}{s-4}e^{e^{468}}N^{\frac{4\delta-s\delta+4s-16}{4}}\bigg).
\end{align*}
Using the inequality 
\[
\left(\frac{24m}{A}\right)^{1/4}\leq N\leq 3\left(\frac{m}{A}\right)^{1/4},
\]
we obtain
\begin{align}
    \bigg|\mathcal{R}_{f, s}^*(m)&- \bigg(\frac{24}{A}\bigg)^{s/4}\mathfrak{S}_{f,s}(m)J_1(m)\bigg| \notag \\
    &\leq \frac{230^s\cdot 730}{9s-146}A^{\frac{9\delta s-146\delta}{292}}m^{\frac{73s-292-9\delta s+146\delta}{292}}+2Ase^{e^{468}}\left(\frac{81m}{A}\right)^{\frac{s-5+2\delta}{4}}\notag\\
    &\quad+\bigg(\frac{24}{A}\bigg)^{s/4} \frac{2^{s+3}}{s-4} e^{e^{468}}{\left(\frac{81m}{A}\right)}^{\frac{4\delta-s\delta+4s-16}{16}}.\label{R*(m)-(2/5)^{s/3}S(m)J_1(m)}
\end{align}
Combining Lemma \ref{Approx 4} and \eqref{R*(m)-(2/5)^{s/3}S(m)J_1(m)}, we have
\begin{align}
\bigg|\int_{\mathfrak{M}}&S_f(\alpha)^s e(-\alpha m) \dd\alpha-\bigg(\frac{24}{A}\bigg)^{s/4}\mathfrak{S}_{f,s}(m)J_1(m)\bigg|\notag\\
&\leq 70(A+4|B|)\pi s \left(\frac{81m}{A}\right)^{\frac{5\delta+s-5}{4}}+\frac{230^s\cdot 730}{9s-146}A^{\frac{9\delta s-146\delta}{292}}m^{\frac{73s-292-9\delta s+146\delta}{292}}\notag\\
    &\quad+2Ase^{e^{468}}\left(\frac{81m}{A}\right)^{\frac{s-5+2\delta}{4}}+\bigg(\frac{24}{A}\bigg)^{s/4} \frac{2^{s+3}}{s-4} e^{e^{468}}{\left(\frac{81m}{A}\right)}^{\frac{4\delta-s\delta+4s-16}{16}}.
\label{f(alpha)-S(m)J_1(m) bound}
\end{align}
Finally, putting together Lemma \ref{J_1(m) approximation bound} and \eqref{f(alpha)-S(m)J_1(m) bound}, for $s\geq 17$ and $m>\frac{A}{24}(e^{e^{467}})^{4/\delta}$,
\begin{align}
\bigg|\int_{\mathfrak{M}}&S_f(\alpha)^s e(-\alpha m) \dd\alpha-\bigg(\frac{24}{A}\bigg)^{\frac{s}{4}}\Gamma\bigg(\frac{5}{4}\bigg)^{s}\Gamma\bigg(\frac{s}{4}\bigg)^{-1}\mathfrak{S}_{f,s}(m) m^{s/4-1}\bigg|\notag\\
 &\leq
70(A+4|B|)\pi s \left(\frac{81m}{A}\right)^{\frac{5\delta+s-5}{4}}+\frac{230^s\cdot 730}{9s-146}A^{\frac{9\delta s-146\delta}{292}}m^{\frac{73s-292-9\delta s+146\delta}{292}}\notag\\
    &\quad+2Ase^{e^{468}}\left(\frac{81m}{A}\right)^{\frac{s-5+2\delta}{4}}+\bigg(\frac{24}{A}\bigg)^{s/4} \frac{2^{s+3}}{s-4}e^{e^{468}}{\left(\frac{81m}{A}\right)}^{\frac{4\delta-s\delta+4s-16}{16}}\notag\\
    &\quad+\left(\frac{24}{A}\right)^{s/4}e^{e^{468}}m^{\frac{s-5}{4}}. \label{4.26 bounds}
\end{align}
The proof now follows immediately from Theorem \ref{thm::minor-arc} and \eqref{4.26 bounds}.
\end{proof}
In Theorem \ref{thm: main Theorem for formula}, we can choose a suitable $\delta$ in terms of $s$ to minimize the exponents of $m$ in the right hand side of \eqref{  1st Theorem Eq}. By examining the first and second terms in the right hand side, we see that $0<\delta<\frac{1}{5}$ is necessary for the main term to dominate the error term. In fact, any $\delta$ within this range satisfies the initial requirement in \eqref{Defining N}. We now proceed to optimize our choice of $\delta$.

\begin{lem}\label{lem::optimal delta} 
One may choose $\delta = \frac{73}{219+9s}$ in the statement of Theorem \ref{thm: main Theorem for formula}. 
\end{lem}
\begin{proof}
Extract the exponents depending on $\delta$ in the right hand side of \eqref{  1st Theorem Eq}. Consider 
\[
G(\delta) = \max \bigg \{ 5\delta - 1, \bigg( 2 - \frac{9s}{73} \bigg )\delta, 2\delta-1, \delta\bigg (1 - \frac{s}{4}\bigg ), -\delta\bigg(\frac{s}{8} -2\bigg) \bigg \}.
\]
Let $\delta_{0}$ be such that $G(\delta_{0}) = \inf \{G(\delta): \delta \in (0,\frac{1}{5})\}$.
When $s\geq 9$, we have 
\[ 
G(\delta) = 
\begin{cases}
    \delta(1-\frac{s}{4}), \text{ if } \delta < 0 \\ \left( 2 - \frac{9s}{73} \right)\delta, \text{ if } \delta \in [0, \delta_{0}] \\ 5\delta - 1, \text{ if } \delta > \delta_{0}.
\end{cases}
\]
Thus $\delta_{0}$ occurs at the intersection of $\left( 2 - \frac{9s}{73} \right)\delta$ and $5\delta - 1$, which implies $\delta_0 =\frac{73}{219+9s}$.
\end{proof}

With Theorem \ref{thm: main Theorem for formula} and Lemma \ref{lem::optimal delta}, we can derive the general versions of Theorem \ref{thm: representations} and Theorem \ref{thm: representations general for deg 4}. By substituting $\delta =\frac{73}{219+9s}$ into Theorem \ref{thm: main Theorem for formula}, we obtain the general version of Theorem \ref{thm: representations general for deg 4}. For Theorem \ref{thm: representations}, the general version is as follows.
\begin{thm}\label{thm: generalized version}
For $m,s \in \mathbb{N}$ and $i=1,2,3$, let $\mathcal{R}_{f_i,s}(m)$ denote the number of representations of $m$ as the sum of $s$ numbers, as given by the target polynomial $f_i$ defined in \eqref{target polynomials}. Then, for any $s\geq 17$ and $m>131\exp\left({\frac{\left(876+36s\right)}{73}e^{467}}\right)$, we have 
\begin{align}
\bigg|\mathcal{R}_{f_i,s}(m)& -\bigg(\frac{24}{A_i}\bigg)^{s/4}\mathfrak{S}_{f,s}(m)\Gamma\bigg(\frac{5}{4}\bigg)^{s}\Gamma\bigg(\frac{s}{4}\bigg)^{-1}m^{s/4-1}\bigg| \notag \\
     &\leq \num{3.5e6}s \left(\frac{9m}{8}\right)^{\frac{9s^2+174s-730}{876+36s}}+\frac{230^s\cdot 730}{9s-146}\cdot {3132}^{\frac{9s-146}{36s+876}}m^{\frac{9s^2+174s-730}{36s+876}}\notag\\
    &\quad+7000se^{e^{468}}\left(\frac{9m}{8}\right)^{\frac{9s^2+174s-949}{876+36s}}+\bigg(\frac{1}{3}\bigg)^{s/4} \frac{2^{s+3}}{s-4}e^{e^{468}}{\left(\frac{9m}{8}\right)}^{\frac{36s^2+659s-3212}{3504+144s}}\notag\\
    &\quad+\left(\frac{1}{3}\right)^{s/4}e^{e^{468}} m^{\frac{s-5}{4}}+10^6 \cdot 11^{s-16} (\log m)^{\frac{s-16}{8}} \left(\frac{9m}{8}\right)^{\frac{s-4}{4} - \frac{73s-1168}{7008+288s} + \frac{s}{4\log\log (m/3132)-8}},
\end{align}
where $A_1 = 72, A_2=580, A_3=3132$, and $\mathfrak{S}_{f,s}(m)$ uniformly satisfy \eqref{Sigma Bound}.    
\end{thm}
\begin{proof}
    For $i=1,2,3$, by writing $f_i$ in the form of \eqref{defn: f(n)}, we obtain the following values: for $f_1, A = 72$ and $B=84$; for $f_2, A=580$ and $B=590$; for $f_3, A = 3132$ and $B = 3186$. Thus,
    \begin{align}\label{min-max A and B}
    \min_{i=1,2,3}{A} = 72, \max_{i=1,2,3}{A} =3132, \textrm{ and }\max_{i=1,2,3}{B} =3186.
    \end{align}
    Substituting \eqref{min-max A and B} into Theorem \ref{thm: main Theorem for formula} and using Lemma \ref{lower bound for S(m)}, we obtain the desired result.
\end{proof}

\bibliographystyle{abbrv}

\bibliography{bibliography}
\end{document}